\pdfoutput=1
\RequirePackage{ifpdf}
\ifpdf 
\documentclass[pdftex]{sigma}
\else
\documentclass{sigma}
\fi

\numberwithin{equation}{section}

\newtheorem{Theorem}{Theorem}[section]
\newtheorem{Corollary}[Theorem]{Corollary}
\newtheorem{Lemma}[Theorem]{Lemma}
\newtheorem{Proposition}[Theorem]{Proposition}
{ \theoremstyle{definition}
\newtheorem{Remark}[Theorem]{Remark} }

\begin{document}

\allowdisplaybreaks

\newcommand{\arXivNumber}{1801.10554}

\renewcommand{\thefootnote}{}

\renewcommand{\PaperNumber}{126}

\FirstPageHeading

\ShortArticleName{Structure Relations of Classical Orthogonal Polynomials}

\ArticleName{Structure Relations of Classical\\ Orthogonal Polynomials in the Quadratic\\ and $\boldsymbol{q}$-Quadratic Variable\footnote{This paper is a~contribution to the Special Issue on Orthogonal Polynomials, Special Functions and Applications (OPSFA14). The full collection is available at \href{https://www.emis.de/journals/SIGMA/OPSFA2017.html}{https://www.emis.de/journals/SIGMA/OPSFA2017.html}}}

\Author{Maurice KENFACK NANGHO~$^{\dag\ddag}$ and Kerstin JORDAAN~$^\S$}

\AuthorNameForHeading{M.~Kenfack Nangho and K.~Jordaan}

\Address{$^\dag$~Department of Mathematics and Applied Mathematics, University of Pretoria,\\
\hphantom{$^\dag$}~Private bag X20 Hatfield, 0028 Pretoria, South Africa}
\Address{$\ddag$~Department of Mathematics and Computer Science, University of Dschang, Cameroon}
\EmailD{\href{mailto:kenfnang@gmail.com}{kenfnang@gmail.com}}

\Address{$^\S$~Department of Decision Sciences, University of South Africa,\\
\hphantom{$^\S$}~PO Box 392, Pretoria, 0003, South Africa}
\EmailD{\href{mailto:jordakh@unisa.ac.za}{jordakh@unisa.ac.za}}

\ArticleDates{Received January 31, 2018, in final form November 13, 2018; Published online November 27, 2018}

\Abstract{We prove an equivalence between the existence of the first structure relation satisfied by a sequence of monic orthogonal polynomials $\{P_n\}_{n=0}^{\infty}$, the orthogonality of the second derivatives $\big\{\mathbb{D}_{x}^2P_n\big\}_{n= 2}^{\infty}$ and a generalized Sturm--Liouville type equation. Our treatment of the ge\-ne\-ralized Bochner theorem leads to explicit solutions of the difference equation [Vinet L., Zhedanov A., \textit{J.~Comput. Appl. Math.} \textbf{211} (2008), 45--56], which proves that the only monic orthogonal polynomials that satisfy the first structure relation are Wilson polynomials, continuous dual Hahn polynomials, Askey--Wilson polynomials and their special or limiting cases as one or more parameters tend to~$\infty$. This work extends our previous result [arXiv:1711.03349] concerning a conjecture due to Ismail. We also derive a second structure relation for polynomials satisfying the first structure relation.}


\Classification{ 33D45; 33C45; 42C05}

\renewcommand{\thefootnote}{\arabic{footnote}}
\setcounter{footnote}{0}

\vspace{-2mm}

\section{Introduction}

A sequence of polynomials $\{P_n(x)\}_{n=0}^{\infty}$, $\deg(P_n(x))=n$, is orthogonal with respect to a positive measure $\mu$ on the real numbers~$\mathbb{R}$, if
\begin{gather*}\int_{S}P_m(x)P_n(x){\rm d}\mu(x)=d_n\delta_{m,n},\qquad m,n\in \mathbb{N},\end{gather*}
where $S$ is the support of $\mu$, $d_n>0$ and $\delta_{m,n}$ the Kronecker delta. A sequence $\{P_n(x)\}$ of monic polynomials orthogonal with respect to a positive measure satisfies a three-term recurrence relation
\begin{gather}\label{e63}
P_{n+1}(x)=(x-a_n)P_{n}(x)-b_nP_{n-1}(x), \qquad n=0,1,2,\ldots,\end{gather}
with initial conditions $P_{-1}(x)= 0$, $P_{0}(x)= 1$, and recurrence coefficients
\begin{gather*} a_n\in\mathbb{R},\qquad n=0,1,2,\dots, \qquad b_{n}>0,\qquad n=1,2,\dots.\end{gather*}

A sequence of monic orthogonal polynomials is classical if the sequence $\{P_n(x)\}$ as well as $D^mP_{n+m}(x)$, $m\in \mathbb{N}$, where $D$ is the usual derivative or one of its extensions (difference operator, $q$-difference operator or divided-difference operator), satisfies a three-term recurrence of the form~(\ref{e63}).

The classical orthogonal polynomials of Jacobi, Laguerre and Hermite are known to be the only polynomials satisfying
\begin{enumerate}\itemsep=0pt
\item[1)] the {\textit{first structure relation}} (cf.~\cite{Al-Chihara1972, Maroni99})
\begin{gather*}
\pi(x)DP_n(x)=\sum_{j=-1}^{1}a_{n,n+j}P_{n+j}(x),\qquad n=1,2,\dots,\qquad a_{n,n-1}\neq 0,
\end{gather*} where $\pi(x)$ is a polynomial of degree at most $2$;
\item[2)] the \textit{second structure relation} (cf.~\cite{Maroni93, Maroni99})
\begin{gather}\label{e57b}
P_n(x)=\sum_{j=-1}^{1}b_{n,n+j}DP_{n+j}(x),\qquad n=0,1,\dots,\qquad b_{n,n+1}=\frac{1}{(n+1)}\neq 0;
\end{gather}
\item[3)] the orthogonality of the sequence of derivatives $\{DP_{n+1}\}_{n=0}^{\infty}$ with respect to $\pi(x) w(x)$ (cf.~\cite{Al-Salam}), where $\pi(x)$ is a polynomial of degree at most~$2$, and $w(x)$ denotes the weight function corresponding to $\{P_n\}_{n=0}^{\infty}$;
\item[4)] a Sturm--Liouville differential equation of the form (cf.~\cite{bochner1929})
\begin{gather*} 
\phi(x)D^2P_n(x)+\psi(x)DP_n(x)+\lambda_n P_n=0,\end{gather*}
where $\phi(x)$, $\psi(x)$ are polynomials with $\deg(\phi(x))\leq 2$, $\deg \psi(x)=1$ and $\lambda_n$ is constant. This result is known as {\it Bochner's theorem} (cf.~\cite{bochner1929}).
\end{enumerate}

The first structure relation, second structure relation and Bochner's theorem have been gene\-ralized to orthogonal polynomials involving the difference and $q$-difference operator (cf.~\cite{AN2006,datta, GaMar1995,Koepf1998,Koepf2001}) and play an important role when studying properties of zeros or connection and linearization problems involving polynomials (see, for example, \cite{Alta, Koepf1998}).

Askey--Wilson polynomials \cite[equation~(1.15)]{Askey-1985}
\begin{gather}\label{e58}
\frac{a^np_n(x;a,b,c,d | q)}{(ab,ac,ad;q)_n}= {}_4 \phi_3 \left(\begin{matrix} q^{-n},abcdq^{n-1},a{\rm e}^{-{\rm i}\theta},a{\rm e}^{{\rm i}\theta}\\ ab,ac, ad\end{matrix}; q,q\right),\qquad x=\cos{\theta},
\end{gather}
and Wilson polynomials \cite[equation~(9.1.1)]{KSL}
\begin{gather}\label{e59}
\frac{W_n\big(x^2;a,b,c,d\big)}{ (a+b)_n(a+c)_n(a+d)_n}= {}_4 F_3 \left(\begin{matrix} -n, n+a+b+c+d-1,a-{\rm i}x,a+{\rm i}x\\ a+b,a+c,a+d\end{matrix}; 1\right)
\end{gather}
do not satisfy structure relations of the type mentioned above but they do satisfy the shift relations (cf.~\cite[equations~(14.1.9) and (9.1.8)]{KSL})
\begin{gather*}
\mathcal{D}_qp_n(x;a,b,c,d\,|\,q)=\frac{2q^\frac{1-n}{ 2}\big(1-q^n\big)\big(1-abcdq^{n-1}\big)}{ 1-q}p_{n-1}\big(x;aq^\frac{1}{ 2},bq^\frac{1}{ 2}, cq^\frac{1}{ 2}, dq^\frac{1}{ 2} \,|\, q\big), \\
\frac{\delta W_n(x^2;a,b,c,d)}{ \delta x^2}=-n(n+a+b+c+d-1)W_{n-1}\left(x^2;a+\frac{1}{ 2},b+\frac{1}{ 2},c+\frac{1}{ 2},d+\frac{1}{ 2}\right),
\end{gather*}
where $\mathcal{D}_q$ is the Askey--Wilson operator (cf.~\cite[p.~33]{Askey-1985}, \cite[equation~(1.16.4)]{KSL}, \cite[equa\-tion (12.1.9)]{Mourad-2005})
\begin{gather}\label{e65}\mathcal{D}_qf(x)=\frac{f\big(q^\frac{1}{ 2}{\rm e}^{{\rm i}\theta}\big)-f\big(q^{-\frac{1}{ 2}}{\rm e}^{{{\rm i}\theta}}\big)}{ \big({\rm e}^{{\rm i}\theta}-{\rm e}^{-{\rm i}\theta}\big)\big(q^\frac{1}{ 2}-q^{-\frac{1}{ 2}}\big)/2},\qquad x=\cos\theta,\end{gather}
and $\delta$ is the Wilson operator
\begin{gather}\label{e65b}
\delta f(x^2)=f\left(\left(x+\frac{{\rm i}}{ 2}\right)^2\right)-f\left(\left(x-\frac{{\rm i}}{2}\right)^2\right).
\end{gather}
Here \begin{gather*}_{s+1} \phi_s \left(
\begin{matrix} a_1,\dots,a_{s+1}\\ b_1,\dots,b_s \end{matrix}
; q,z \right)
=\sum_{k=0}^{\infty}\frac{(a_1;q)_k\cdots(a_{s+1};q)_k}{(b_1;q)_k\cdots(b_s;q)_k}\frac{z^k}{(q;q)_k},\end{gather*} with
\begin{gather*}
 (a;q)_0=1, \qquad (a;q)_k= \prod_{j=0}^{k-1}\big(1-aq^j\big),\qquad k=1,2,\dots,\end{gather*} and
\begin{gather*}_{s+1} F_s \left(
\begin{matrix} a_1,\dots,a_{s+1}\\ b_1,\dots,b_s \end{matrix}
; z \right)=\sum_{k=0}^{\infty}\frac{(a_1)_k\cdots (a_{s+1})_k}{(b_1)_k\cdots (b_{s})}\frac{z^k}{k!},\end{gather*}
with $(a)_0=1$ and $(a)_k= \prod\limits_{j=0}^{k-1}(a+j)$, $k=1,2,\dots$.

Since the appearance of Askey--Wilson and Wilson polynomials in the early 1980's (cf.~\cite{Andrew-Askey, Askey-1985}), many authors have studied these polynomials (see, for example, \cite{ARS1995, CSM, Mourad2003, Mourad-2005, Koorn,Magnus1988, VinetZhedanov}). Ismail considered the problem of the first structure relation for Askey--Wilson polynomials in the conjecture \cite[Conjecture~24.7.9]{Mourad-2005}. In \cite[Corollary~3.3]{MK2017}, we completed the conjecture by proving that a sequence of monic orthogonal polynomials satisfies the first structure relation
\begin{gather}\label{structureAW}
\pi(x)\mathcal{D}_q^2P_n(x)=\sum_{j=-2}^{2}a_{n,n+j}P_{n+j}(x),\qquad a_{n,n-2}\neq 0, \qquad x=\cos\theta,
\end{gather}
where $\pi$ is a polynomial of degree at most 4, if and only if $P_n(x)$ is an Askey--Wilson polynomial up to a multiplicative constant or a subcase of Askey--Wilson polynomials, including limiting cases as one or more of the parameters tend to $\infty$ (cf.~\cite{Mourad2003}). This result holds for orthogonal polynomials of the variable $x=\frac{{\rm e}^{-{\rm i}\theta}+{\rm e}^{{\rm i}\theta}}{ 2}= \cos\theta$ which can also be written as $x(s)=\frac{q^{-s}+q^s}{2}$, ${\rm e}^{{\rm i}\theta}=q^s$.

Even though Askey--Wilson polynomials (\ref{e58}) are a basic hypergeometric analog of the Wilson polynomials~(\ref{e59}) (cf.~\cite[p.~188]{ARS1995}), the coefficients in the analog of (\ref{structureAW}) for the Wilson operator
\begin{gather}\label{structureW}
\pi(x)\frac{\delta^2 P_n(x)}{ \delta^2x^2}=\sum_{j=-2}^{2}a_{n,n+j}P_{n+j}(x),\qquad a_{n,n-2}\neq0,
\end{gather}
as well as its solutions can not easily be deduced from those of Askey--Wilson polynomials. It therefore is necessary to consider the Ismail conjecture for the Wilson variable $x(z)=z^2$ $(z={\rm i}s$, ${\rm i}^2=-1)$, or, more generally, for the quadratic and $q$-quadratic variable (cf.~\cite{niki1991})
\begin{gather}\label{e60}
x(s)= \begin{cases}
c_1 q^{-s}+c_2 q^{s}+c_3&\text{if} \ q\neq 1,\\
c_4 s^2+c_5s+c_6& \text{if} \ q=1,
\end{cases}
\end{gather}
where $c_1\neq 0$ and $c_4\neq 0$. This problem of characterizing the orthogonal polynomials of the variable $x(s)$ whose derivatives satisfy a generalized first structural relation is a generalization of the Askey problem (cf.\ \cite[p.~69]{Al-Chihara1972}).

The aim of this paper is to use generalizations of Bochner's theorem in \cite{Mourad2003,VinetZhedanov} (see also~\cite{Grunbaum}) for classical orthogonal polynomials of the quadratic and $q$-quadratic variable $x(s)$ defined in~(\ref{e60}) to obtain a generalized first structure relation for classical orthogonal polynomials of the variable $x(s)$ of the form
\begin{gather*}
\pi(x)\mathbb{D}_{x}^2P_n(x)=\sum_{j=-r}^{t}a_{n,n+j}P_{n+j}(x),\qquad n=1,2,\dots,
\end{gather*} where $\mathbb{D}_{x}$ is the divided-difference operator (cf.~\cite{foupouagnigni2008})
\begin{gather}\label{ddop}
\mathbb{D}_{x} f(x(s))=\frac{f\big(x\big(s+\frac{1}{ 2}\big)\big)-f\big(x\big(s-\frac{1}{2}\big)\big)}{ x\big(s +\frac{1}{ 2}\big)-x\big(s-\frac{1}{ 2}\big)}.
\end{gather}

This work is organized as follows: In Section \ref{2}, we derive explicit solutions for the second-order difference equation \cite[equation~(1.3)]{VinetZhedanov}
\begin{gather} A(s)P_n(x(s+1))+B(s)P_n(x(s))+C(s)P_n(x(s-1))=\lambda_nP_n(x(s)),\label{vinet} \end{gather} where $A(s)$, $B(s)$, $C(s)$ are some functions of the discrete argument $s$, and $P_0=1$, shown to characterize polynomials of the variable $x(s)$ in \cite{VinetZhedanov} by Vinet and Zhedanov. In Section~\ref{3}, we will show that this generalized Bochner theorem (cf.~\cite{VinetZhedanov}) is related to the generalized Askey problem and we will characterize Wilson and Askey--Wilson polynomials, and subcases, including limiting cases, as the only monic orthogonal polynomials satisfying the first structure relation
\begin{gather*}
\pi(x(s))\mathbb{D}_{x}^2 P_n(x(s))=\sum_{j=-2}^{2}a_{n,n+j}P_{n+j}(x(s)),\qquad a_{n,n-2}\neq 0, \qquad n=2,3,\dots,
\end{gather*}
where $\pi(x)$ is a polynomial of degree at most four and $x(s)$ is given by (\ref{e60}). We then derive the second structure relation
\begin{gather}\label{SecondSR}
P_n(x(s))=\sum_{j=-2}^{2}b_{n,n+j}\mathbb{D}_{x}^2P_{n+j}(x(s)),
\end{gather}
for classical orthogonal polynomials of the variable $x(s)$ and conclude the section by connecting the second structure relation (\ref{SecondSR}) to that of Costas-Santos and Marcell\'{a}n \cite[p.~118]{CSM} \begin{gather}\label{macellan}
\mathcal{M}P_n(x(s))=e_n\mathbb{D}_{x} P_{n+1}(x(s))+f_n\mathbb{D}_{x} P_n(x(s))+g_n\mathbb{D}_{x} P_{n-1}(x(s)),\\
\mathcal{M} f(s)=\frac{f\big(s+\frac{1}{ 2}\big)+f\big(s-\frac{1}{ 2}\big)}{ 2}.\nonumber \end{gather}

In Section~\ref{4} we compute coefficients of (\ref{structureW}) for the Wilson polynomials as well as those of the second structure relation (\ref{SecondSR}) for the Wilson polynomials and Askey--Wilson polynomials.
\section{Preliminaries and notation}
Let us recall some basic results and notations. $x(s)$ given by (\ref{e60}) satisfies (cf.~\cite{ARS1995})
\begin{gather*}
x(s+n)-x(s)=\gamma_n {\nabla} x_{n+1}(s),\qquad
\frac{x(s+n)+x(s)}{ 2}=\alpha_n x_n(s)+\beta_n, \end{gather*} for $n=0,1,\dots,$ with
\begin{gather*} x_\mu(s)=x\left(s+\frac{\mu}{ 2}\right),\qquad \mu\in{\mathbb{C}}, \end{gather*} where
${\mathbb{C}}$ is the set of complex numbers and ${\nabla}$ is the backward difference operator ${\nabla} f(s):=f(s)-f(s-1)$.
The sequences $(\alpha_n)$, $(\beta_n)$, $(\gamma_n)$ are given explicitly by (cf.~\cite{ARS1995}), $\alpha_1=\alpha$, $\beta_1=\beta$,
\begin{gather*}
\alpha_n=1,\qquad \beta_n=\beta n^2,\qquad \gamma_n=n,\qquad \alpha=1,\qquad \beta=\frac{c_4}{ 4} \qquad \text{for}\quad q=1, \end{gather*}
and
\begin{gather*}
\alpha_n=\frac{q^\frac{n}{ 2}+q^{-\frac{n}{ 2}}}{ 2},\qquad \beta_n=\frac{\beta(1-\alpha_n)}{ 1-\alpha},\qquad \gamma_n=\frac{q^\frac{n}{ 2}-q^{-\frac{n}{ 2}}}{ q^\frac{1}{2}-q^{-\frac{1}{ 2}}},\\
\alpha=\frac{{q^\frac{1}{ 2}+q^{-\frac{1}{2}}}}{ 2},\qquad \beta=-c_3\frac{\big(\sqrt{q}-1\big)^2}{2\sqrt{q}},\qquad \text{for}\quad q\neq 1.
\end{gather*}
The following hold (cf.~\cite{foupouagnigni2008}):
\begin{gather}
\mathbb{D}_{x}(fg)=\mathbb{D}_{x}(f)\mathbb{S}_{x}(g)+\mathbb{S}_{x}(f)\mathbb{D}_{x}(g)\label{e16},\\
\mathbb{S}_{x}(fg)=\mathbb{S}_{x}(f)\mathbb{S}_{x}(g)+U_2\mathbb{D}_{x}(f)\mathbb{D}_{x}(g)\label{e16a},\\
\mathbb{D}_{x}\mathbb{S}_{x} =\alpha\mathbb{S}_{x} \mathbb{D}_{x}+U_1 \mathbb{D}_{x}^2\label{e21a},\\
\mathbb{S}_{x}^2= U_1\mathbb{S}_{x}\mathbb{D}_{x}+\alpha U_2\mathbb{D}_{x}^2+\mathbb{I},\label{e16b}
\end{gather}
where
\begin{gather*}
U_1(x(s))=\big(\alpha^2-1\big)x(s)+\beta (\alpha+1),\\
U_2(x(s))=\left(\frac{x\big(s+\frac{1}{ 2}\big)-x\big(s-\frac{1}{ 2}\big)}{ 2}\right)^2=\big(\alpha^2-1\big)x(s)^2+2\beta (\alpha+1)x(s)+C_x,
\end{gather*}
with
\begin{gather*}C_x=\frac{c_{5}^{2}}{4}-c_{4}c_{6}, \qquad \text{for} \quad q=1 \qquad \text{and} \qquad C_x={\frac {( q-1) ^{2} \big(c_3^{2}- 4 c_{1}c_{2} \big) }{4q}}, \qquad \text{for} \quad q\neq 1.
\end{gather*}
Note that $\mathbb{I}(f)=f$ and $\mathbb{S}_{x}$ the averaging operator
\begin{gather*}\mathbb{S}_{x} f(x(s))=\frac{f\big(x\big(s+\frac{1}{ 2}\big)\big)+f\big(x\big(s-\frac{1}{2}\big)\big)}{ 2},\end{gather*}
which is a generalization of \cite[equation~(12.1.21)]{Mourad-2005}. Taking ${\rm e}^{{\rm i}\theta}=q^s$, the Askey--Wilson operator~(\ref{e65}) reads as
\begin{gather*} \mathcal{D}_qf(x(s))=\frac{f\big(x\big(s+\frac 12\big)\big)-f\big(x\big(s-\frac{1}{2}\big)\big)}{x\big(s+\frac{1}{2}\big)-x\big(s-\frac{1}{2}\big)}=\mathbb{D}_{x} f(x(s)), \qquad x(s)=\frac{q^{-s}+q^s}{2}.
\end{gather*} The Wilson operator~(\ref{e65b}) is connected to the divided-difference operator (\ref{ddop}) as follows:
\begin{gather*}
\frac{\delta f\big(s^2\big)}{ \delta s^2}=\frac{f\big({-}\big({\rm i}s-\frac{1}{ 2}\big)^2\big)-f\big({-}\big({\rm i}s+\frac{1}{ 2}\big)^2\big)}{ -\big({\rm i}s-\frac{1}{ 2}\big)^2 +\big({\rm i}s+\frac{1}{ 2}\big)^2}=-\mathbb{D}_{x} f(-x({\rm i}s)),\qquad x(z)=z^2,\qquad z={\rm i}s.
\end{gather*}

\section{Generalized Bochner theorem }\label{2}
In this section, using properties of the divided-difference operator $\mathbb{D}_{x}$ and the averaging opera\-tor~$\mathbb{S}_{x}$, we discuss generalized versions of Bochner's theorem in~\cite{Mourad2003} and~\cite{VinetZhedanov} and derive explicit expressions for the polynomial solutions characterized by the results.
\begin{Lemma}\label{Prop4} Polynomial solutions $P_n(x(s))$, $\deg(P_n(x(s)))=n$, of the Sturm--Liouville type equation
\begin{gather}
\phi(x)\mathbb{D}_{x}^2y(x)+\psi(x)\mathbb{S}_{x}\mathbb{D}_{x} y(x)+\lambda y(x)=0, \label{e1.10.1}
\end{gather}
where $\phi(x)=\phi_2x^2+\phi_1x+\phi_0$ and $\psi(x)=\psi_1x+\psi_0$ are polynomials of degree at most two and one, can be expanded as
\begin{gather}\label{e39}
P_n(x(s))=\sum_{k=0}^{n} d_k\prod_{j=0}^{k-1}[x(s)-x(\eta+j)],
\end{gather}
where $\eta$ is a complex number such that $\sigma(x(\eta))=0$
where \begin{gather}
\label{sigma} \sigma(x(s))=\phi(x(s))-\frac{{\nabla} x_1(s)}{ 2}\psi (x(s)),
\end{gather} and $d_k$ is solution to the first-order recurrence relation
\begin{gather}\label{e39b}
\left(\lambda+\gamma_k\gamma_{k-1}\phi_2+\gamma_k\alpha_{k-1}\phi_1\right)d_k\\
\quad{}+\left(\gamma_{k}\gamma_{k+1}\left(\phi_2\left(x(\eta+k)+x(\eta)\right)+\phi_1-\frac{\psi_1{\nabla} x_1(\eta)}{ 2}\right)+\alpha_k\gamma_{k+1}\psi(x(\eta+k))\right)d_{k+1}=0,\nonumber
\end{gather} with $\lambda=-\gamma_{n}\gamma_{n-1}\phi_2-\gamma_n\alpha_{n-1}{\psi_1}$.
\end{Lemma}
\begin{proof}Write
\begin{gather}\label{e39a}
w_k(x(s), \eta)=\prod_{j=0}^{k-1}[x(s)-x(\eta+j)], \qquad k>1\qquad \text{and} \qquad w_0(x(s))\equiv 1,
\end{gather}
and obtain by direct computation
\begin{gather}
\mathbb{D}_{x} w_k(x(s), \eta) = \gamma_kw_{k-1}\left(x(s), \eta+\frac{1}{ 2}\right),\label{e40}\\
\mathbb{S}_{x} w_k(x(s), \eta) = \alpha_kw_k\left(x(s), \eta-\frac{1}{ 2}\right)-\frac{\gamma_k {\nabla} x(\eta)}{ 2}w_{k-1}\left(x(s), \eta+\frac{1}{ 2}\right), \label{e41}\\
x(s)w_{k}(x(s), \eta) = w_{k+1}(x(s), \eta-1)+x(\eta-1)w_k(x(s),\eta),\label{e42} \\
w_k(x(s), \eta)= w_k(x(s), \eta+1)+(x(\eta+k)-x(\eta))w_{k-1}(x(s), \eta+1).\label{e43}
\end{gather}

Next, take $P_n(x(s))=\sum\limits_{k=0}^{\infty} d_kw_k(x(s), \eta)$ with $\phi(x(s))=\phi_2x(s)^2+\phi_1x(s)+\phi_0$ and $\psi(x(s))=\psi_1x(s)+\psi_0$ in (\ref{e1.10.1}) and use (\ref{e40})--(\ref{e42}) for simplification to derive
\begin{gather*}
\sum_{k=0}^{\infty}\!\bigg\{\!\gamma_k\!\left[\gamma_{k-1}\left(\phi_2\left(x(\eta-1)+x(\eta)\right)+\phi_1\right)+\!\left(\!\alpha_{k-1}\psi_1 x(\eta-1)-\psi_1 \frac{\gamma_{k-1}\nabla x_1(\eta)}{ 2}+\alpha_{k-1}\psi_0\!\right)\right]\\\qquad{}\times w_{k-1}(x(s), \eta) +\left(\gamma_k\gamma_{k-1}\phi_2+\gamma_k\alpha_{k-1}\psi_1\right)w_k(x(s), \eta-1)+\lambda\,w_k(x(s),\eta)+\sigma (x(\eta))\\
\qquad{} \times \gamma_k\gamma_{k-1}w_{k-2}(x(s), \eta+1)\bigg\}d_k=0.
\end{gather*}

Now, use the fact that $\sigma (x(\eta))=0$ and the relation (\ref{e43}) to obtain
\begin{gather*}\sum_{k=0}^{\infty}\left(A_kd_k+B_kd_{k+1}\right)w_k(x(s), \eta)=0,\end{gather*}
and equate coefficients of $w_k$ to obtain the two-term recurrence relation
\begin{gather*}A_kd_k+B_kd_{k+1}=0, \qquad k\geq 0,\end{gather*}
where
\begin{gather*}
A_k=\lambda+\gamma_k\gamma_{k-1}\phi_2+\gamma_k\alpha_{k-1}\phi_1, \nonumber \\
B_k= \gamma_{k}\gamma_{k+1}\left[\phi_2\left(x(\eta+k)+x(\eta)\right)+\phi_1-\frac{\psi_1\nabla x_1(\eta)}{ 2}\right]+\alpha_k\gamma_{k+1}\psi(x(\eta+k)).
\end{gather*}By assumption, $P_n$ is a polynomial of degree $n$, that is $d_k=0$, $k\geq n+1$. Hence taking $k=n$ we obtain $\lambda=-\phi_2\gamma_n\gamma_{n-1}-\psi_1\gamma_n\alpha_{n-1}$. Taking into account this expression of $\lambda$ the required relation is obtained.
\end{proof}

\begin{Remark}\label{Prop5} The explicit expressions of $w_k(x(s),\eta)$ in (\ref{e39a}) for the corresponding lattices $x(s)$ are provided in the following table:
\begin{center}
\begin{tabular}{|c|c|}
\hline
Representation of $w_k(x(s),\eta)$ & On the lattice $x(s)$\\
\hline
$\big({-}\frac{q^{-\eta}}{ 2}\big)^kq^{-{k\choose 2}}\big(q^{\eta}q^{-s};q\big)_k \big(q^{\eta}q^{s};q\big)_k$ \tsep{4pt}\bsep{4pt} & $x(s)=\frac{q^{-s}+q^s}{ 2}$ \\
\hline
$\big({-}c_1q^{-\eta}\big)^kq^{-{k\choose 2}}\big(q^{\eta}q^{-s};q\big)_k\big(\frac{c_2}{ c_1}q^{\eta}q^{s};q\big)_k $ \tsep{4pt}\bsep{4pt} & $x(s)=c_1q^{-s}+c_2q^s+c_3$, $c_1\neq 0$ \\
\hline
$\big({-}c_1q^{-{\eta}}\big)^kq^{-{k\choose 2}}\big(q^{\eta}q^{-s};q\big)_k $\tsep{4pt}\bsep{4pt} & $x(s)=c_1q^{-s}+c_3$ \\
\hline
$(-c_4)^k\big(s+\frac{c_5}{ c_4}+\eta\big)_k(-s+\eta)_k $\tsep{3pt}\bsep{3pt} & $x(s)=c_4s^2+c_5s+c_6$, $c_4\neq 0$ \\
\hline
$(-c_5)^k(-s+\eta)_k $\tsep{2pt}\bsep{2pt} & $x(s)=c_5s+c_6$ \\
\hline
\end{tabular}
\end{center}
\end{Remark}

When the function $\sigma (x(s))$ (with $x(s)=c_1q^{-s}+c_2q^s$, $c_1c_2\neq 0$) happens to be of the form $C(q^s)^m$, $m=0,1,\dots$, it has no zeros and therefore Lemma~\ref{Prop4} can no longer be used for expanding polynomial solutions of the Sturm--Liouville type equation (\ref{e1.10.1}). We will see later that this problem arises for the special case of~(\ref{e1.10.1}) when
\begin{gather}\label{symcase}
\phi(x)=\phi_2x^2+\phi_0 \qquad \text{and} \qquad \psi(x)=\psi_1x.
\end{gather}
In \cite{kfk2017}, a method for solving (\ref{e1.10.1}), when $\phi$ and $\psi$ are of the form (\ref{symcase}), was developed using the generalized form of the basis $\rho_n(x)=\big(1+{\rm e}^{2{\rm i}\theta}\big)\big({-}q^{2-n}{\rm e}^{2{\rm i}\theta};q^2\big)_{n-1}{\rm e}^{-{\rm i}n\theta}$, $x=\cos\theta$ (cf.\ \cite[equation~(20.3.8)]{Mourad-2005}) on the lattice $x(s)=c_1q^{-s}+c_2q^s$. This result can be written as:

\begin{Lemma}[{\cite[Theorem~13]{kfk2017}}]\label{Prop6} On $q$-quadratic lattices $x(s)=c_1q^{-s}+c_2q^s$, polynomial solutions~$P_n$ of~\eqref{e1.10.1} when $\phi$ and $\psi$ are of the form
\begin{gather*}\phi(x(s))=\phi_2x(s)^2+\phi_0, \qquad \psi(x(s))=\psi_1x(s),
\end{gather*}
can be expanded as
\begin{gather*}P_n(x(s))=\sum_{k=0}^{[\frac{n}{ 2}]}d_{n-2k}K_{n-2k}(x(s)),\end{gather*}
where
\begin{gather*}K_j(x(s))=\big(c_1q^{-s}\big)^j\left(1+\frac{c_2}{ c_1}q^{2s}\right)\left(-\frac{c_2}{ c_1}q^{-j+2}q^{2s};q^2\right)_{j-1}, \qquad K_0(x(s))=1,\qquad j\geq 1,\end{gather*}
and $d_j$ is given by the two-term recurrence relation
\begin{gather*}
\big(\gamma_j\big(\phi_2\gamma_{j-1}+(\gamma_j-\alpha \gamma_{j-1})\psi_1\big)+\lambda\big)d_j+\gamma_{j+2}\gamma_{j+1}\phi\big({\rm i}\sqrt{c_1c_2}\big(q^\frac{j}{ 2}-q^{-\frac{j}{ 2}}\big)\big)d_{j+2}\\
\qquad{} +\psi_1\gamma_{j+2}\big(\gamma_j\big({\rm i}\sqrt{c_1c_2}\big(q^\frac{j+1}{ 2}-q^{-\frac{j+1}{ 2}}\big)\big)^2-\alpha\gamma_{j+1}\big({\rm i}\sqrt{c_1c_2}\big(q^\frac{j}{ 2}-q^{-\frac{j}{ 2}}\big)\big)^2\big) d_{j+2}
=0,
\end{gather*} with the coefficient $\lambda=-\gamma_{n}\gamma_{n-1}\phi_2-\gamma_n\alpha_{n-1}{\psi_1}$.
\end{Lemma}
Ismail \cite{Mourad2003} gave the following generalization of Bochner's theorem for Askey--Wilson polynomials where $\mathcal{S}_q$ (cf.\ \cite[equation~(12.1.21)]{Mourad-2005}) is the restriction of the averaging operator \begin{gather*}
\mathbb{S}_{x} f(x(s))=\frac{f\big(x\big(s+\frac{1}{ 2}\big)\big)+f\big(x\big(s-\frac{1}{ 2}\big)\big)}{ 2}
\end{gather*} to functions of the variable $x=\cos\theta=\frac{q^{-s}+q^s}{2}$, $q^s={\rm e}^{{\rm i}\theta}$.
\begin{Theorem}[{\cite[Theorem~3.1]{Mourad2003}}]\label{th2} The Sturm--Liouville type equation
\begin{gather}\label{qsturm}
\phi(x)\mathcal{D}_q^2y(x)+\psi(x)\mathcal{S}_q\mathcal{D}_qy(x)+\lambda_ny(x)=0, \qquad x=\cos\theta,
\end{gather}
where $\phi$ and $\psi$ are polynomials of degree at most $2$ and $1$, has a polynomial solution~$P_n(x)$ of degree $n=1,2,3,\dots$ if and only if $P_n(x)$ is a multiple of the Askey--Wilson polynomial $p_n(x;a,b,c,d\,|\,q)$ for some parameters $a$, $b$, $c$, $d$, including limiting cases as one or more of the parameters tend to~$\infty$.
\end{Theorem}
In order to obtain all the explicit solutions characterized by (\ref{qsturm}) we use the following scheme:
\begin{enumerate}\itemsep=0pt
\item Since $\phi(x(s))$ is at most quadratic and $\psi(x(s))$ is linear, we write
\begin{gather*} \phi(x(s))=\phi_2x(s)^2+\phi_1x(s)+\phi_0, \qquad \psi(x(s)=\psi_1x(x(s))+\psi_0.\end{gather*} Substituting $\phi(x(s))$, $\psi(x(s))$ and $x(s)=\frac{q^{-s}+q^s}{ 2}$ into (\ref{sigma}) we obtain, for $X=q^s$, $\sigma(x(s))=\frac{P(X)}{X^2}$, where
\begin{gather}
8\sqrt{q}P(X) = \big( 2 \phi_2\sqrt {q}-q\psi_1+\psi_1 \big) {X}^{4}+ \big( 4 \sqrt {q}\phi_1-2 q\psi_0+2 \psi_0 \big) {X}^{3}\nonumber\\
\quad{} + \big( 8\sqrt {q}\phi_0+4 \phi_2\sqrt {q} \big) {X}^{2}+ \big( 4 \sqrt {q}\phi_1+2 q\psi_0-2 \psi_0 \big) X+2 \phi_2\sqrt {q}+q\psi_{{1}}-\psi_1. \label{e38a}
\end{gather}
\item Suppose $P(X)$ is of degree 4, that is $\psi_1\neq \frac{2\phi_2 \sqrt{q}}{ q-1}$, and write
\begin{gather*}P(X)=C(X-a)(X-b)(X-c)(X-d),\qquad a,b,c,d\in \mathbb{C}.\end{gather*}
\item Expand the factorized form of $P(X)$ and identify coefficients of $X^i$, $i=0, 1, 2, 3, 4$ with those in (\ref{e38a}) to obtain a system of five equations.
\item Solve the system with unknowns $\phi_2$, $\phi_1$, $\phi_0$, $\psi_1$ and $\psi_0$ to obtain polynomial coefficients of (\ref{qsturm}).
\item If one of the $a$, $b$, $c$, $d$, say $a$, is different from $0$, then use Lemma \ref{Prop4}, with $q^{\eta}=a$, to obtain polynomial solution of~(\ref{qsturm}) of the form
\begin{gather}\label{e46}
\sum_{k=0}^{n}d_kw_k(x,\eta),
\end{gather}
where $\frac{d_{k+1}}{ d_k}$ is given by (\ref{e39b}).
\item Iterate (\ref{e39b}) to obtain $d_k$ and use it as well as the representation of $w_k$, for the lattice $x(s)=\frac{q^{s}+q^s}{ 2}$ (see Remark \ref{Prop5}), to obtain the basic hypergeometric representation of (\ref{e46}).
\item If none of $a$, $b$, $c$ and $d$, is different from zero (which corresponds to $\phi(x)=\phi_2x^2+\phi_0$ and $\psi (x)=\psi_1x$), use Lemma~\ref{Prop6} to solve~(\ref{qsturm}).
\end{enumerate}
At the end of the day, one has the following:
\begin{enumerate}\itemsep=0pt
\item If P(X) is of degree 4,
\begin{itemize}\itemsep=0pt
\item If $P(X)=CX^4$, then polynomial coefficients of (\ref{qsturm}) are up to a multiplicative factor equal to
\begin{gather*}
\phi(x(s))=2x(s)^2-1, \qquad \psi(x(s))=-\frac{4\sqrt{q}}{ q-1}x(s),
\end{gather*}
and the corresponding polynomial, with $q^s={\rm e}^{{\rm i}\theta}$, is up to a multiplicative factor equal to
\begin{gather*}
\sum_{k=0}^{\left[\frac{n}{ 2}\right]}d_{n-2k}K_{n-2k}(x),
\end{gather*}
with
\begin{gather*}\frac{d_{n-2k}}{ d_{n-2(k-1)}}=-\frac{1}{ 4}\frac{\big(1-q^{-n-1}q^{2k}\big)\big(1-q^{-n-2}q^{2k}\big)}{ 1-q^{2k}}q^{n+1},\\ d_{n-2k}=d_n\frac{\big(q^{-n};q\big)_{2k}}{ \big(q^2;q^2\big)_k}\left(-\frac{1}{ 4}q^{n+1}\right)^{k}.\end{gather*}
Taking $d_n=1$, we obtain after straightforward computation
\begin{gather*}
\sum_{k=0}^{[\frac{n}{ 2}]}\frac{\big(q^{-n},q\big)_{2k}}{ \big(q^2, q^2\big)_k}\left(-\frac{q^{n+1}}{ 4}\right)^{k}K_{n-2k}(x)=2^{-n}H_n(x\,|\,q),
\end{gather*}
where $H_n(x\,|\,q)$ is the continuous $q$-Hermite polynomial \cite[equation~(14.26.1)]{KSL}.
\item If $P(X)=C(X-a)X^3$, $a\neq 0$, then coefficients of (\ref{qsturm}) are up to a multiplicative factor equal to
\begin{gather*}
\phi(x(s)) = 2x(s)^2-ax(s)-1, \qquad \psi(x(s)) = -\frac{4\sqrt{q}}{ q-1}x(s)+\frac{2a\sqrt{q}}{ q-1}.
\end{gather*}
So,
\begin{gather*}\frac{d_{k+1}}{ d_k}= -2 {\frac {{q}^{k}qa \big( 1-{q}^{k}{q}^{-n} \big) }{ \big( 1-{q}^{k}q \big) }},
d_k={\frac { \big( {q}^{-n};q\big)_k ( -2 qa ) ^{k}{q}^{{k\choose 2}}}{ ( q;q)_k }}d_0,\end{gather*}
and the basic hypergeometric representation of the corresponding polynomial, with $q^s={\rm e}^{{\rm i}\theta}$, is
\begin{gather*}_3 \phi_2 \left(
\begin{matrix} q^{-n}, a{\rm e}^{-{\rm i}\theta},a{\rm e}^{{\rm i}\theta}\\ 0,0
\end{matrix}
; q,q\right)=a^nH_n(x;a\,|\,q),
\end{gather*}
where $H_n(x;a|q)$ is the continuous big $q$-Hermite polynomial \cite[equation~(14.18.1)]{KSL}.
\item If $P(X)=C(X-a)(X-b)X^2$, $ab\neq0$,
\begin{gather*}
\phi(x(s))=2x(s)^2-(a+b)x(s)+ab-1, \\
 \psi(x(s))=-\frac{4\sqrt{q}}{ q-1}x(s)+\frac{2\sqrt{q}(a+b)}{ q-1},\\
\frac{d_{k+1}}{ d_k}=-2 {\frac {{q}^{k}qa \big( 1-{q}^{-n}{q}^{k} \big) }{ \big( 1-{q}^{k}q \big) \big(1- {q}^{k}ab \big) }}, \qquad d_k={\frac { \big( {q}^{-n};q \big)_k ( -2 qa ) ^{k}{q}^{{k\choose 2}}}{( q;q)_k ( ab;q)_k }}d_0,\end{gather*}
and the basic hypergeometric representation of the corresponding polynomial, with $q^s={\rm e}^{{\rm i}\theta}$, is
\begin{gather*}_3 \phi_2 \left(
\begin{matrix} q^{-n}, a{\rm e}^{-{\rm i}\theta},a{\rm e}^{{\rm i}\theta}\\ ab,0
\end{matrix}
; q,q\right) =\frac{a^n}{ (ab;q)_n}Q_n(x; a, b\,|\,q),
\end{gather*}
where $Q_n(x; a, b\,|\,q)$ is the Al-Salam Chihara polynomial \cite[equation~(14.8.1)]{KSL}.
\item If $P(X)=C(X-a)(X-b)(X-c)X$, $abc\neq0$,
\begin{gather*}
\phi(x(s))=2x(s)^2-(abc+a+b+c)x(s)+ab+ac+bc-1,\\
\psi(x(s))=-\frac{4\sqrt{q}}{ q-1}x(s)-\frac{2\sqrt{q}(abc-a-b-c)}{ q-1},
\\
\frac{d_{k+1}}{ d_k}=-2 {\frac {{q}^{k}qa \left( 1-{q}^{-n}{q}^{k} \right) }{ \left(1- {q}^{k}q \right) \left( 1-{q}^{k}ac \right) \left(1- {q}^{k}ab \right) }},\qquad d_k={\frac { \left( {q}^{-n};q \right)_k \left( -2 qa \right) ^{k}{q}^{{k\choose 2}}}{ \left( q;q
\right)_k \left( ab;q \right)_k \left( ac;q \right)_k }}d_0,
\end{gather*}
and the basic hypergeometric representation of the corresponding polynomial is
\begin{gather*}_3 \phi_2 \left(
\begin{matrix} q^{-n}, a{\rm e}^{-{\rm i}\theta},a{\rm e}^{{\rm i}\theta}\\ ab,ac
\end{matrix}; q,q\right)=\frac{a^np_n(x;a,b,c\,|\,q)}{ (ab, ac;q)_n},
\end{gather*}
where $p_n(x;a,b,c|q)$ is the continuous dual $q$-Hahn polynomial \cite[equation~(14.3.1)]{KSL}.
\item If $P(X)=C(X-a)(X-b)(X-c)(X-d)$, $abcd\neq0$,
\begin{gather*}
\phi(x(s))=2(abcd+1)x(s)^2-(abc+abd+acd+bcd+a+b+c+d)x(s)\\
\hphantom{\phi(x(s))=}{} -abcd+ab+ac+ad+bc+bd+cd-1,\\
\psi(x(s))=4 {\frac {\sqrt {q} ( abcd-1) }{q-1}}x(s)-2 {\frac {\sqrt {q} ( abc+abd+acd+bcd-a-b-c-d) }{q-1}},\\
\frac{d_{k+1}}{ d_k}=-2 {\frac {{q}^{k}qa \big(1-{q}^{k}{q}^{-n} \big) \big( 1-{q}^{n-1}abcd{q}^{k} \big) }{
\big(1- {q}^{k}q \big) \big(1- {q}^{k}ad\big) \big( 1-{q}^{k}ac \big) \big( 1-{q}^{k}ab\big) }},\\
d_k={\frac {( -2 aq ) ^{k} \big( {q}^{-n};q \big)_k {q}^{{k\choose 2}} \big( {q}^{n-1}abcd;q \big)_k }{( q;q)_k ( ad;q)_k ( ac;q )_k( ab;q )_k }}d_0,
\end{gather*}
and the basic hypergeometric representation of the corresponding polynomial is
\begin{gather*}_4 \phi_3 \left(
\begin{matrix} q^{-n}, abcdq^{n-1}, a{\rm e}^{-{\rm i}\theta},a{\rm e}^{{\rm i}\theta}\\ ab, ac, ad
\end{matrix}; q,q \right)=\frac{a^np_n(x;a,b,c,d;q)}{ (ab,ac,ad;q)_n},
\end{gather*}
where $p_n(x;a,b,c,d\,|\,q)$ is the Askey--Wilson polynomial \cite[equation~(14.1.1)]{KSL}.
\end{itemize}

In the following items, we are going to consider cases for which degree of $P(X)<4$. For each of them, after giving the factorized form of P(X), we will follow steps 3 and 4 of the scheme to look for $\phi_2$, $\phi_1$, $\phi_0$, $\psi_1$ and $\psi_0$. Then we will follow steps~5 and~6 for degree of $P(X)=1,2,3$ and steps~5 and~7 for degree $P(X)=0$ to obtain the corresponding polynomials system.
\item If $P(X)$ is of degree 3 and $P(X)=C(X-a)(X-b)(X-c)$ with none of~$a$, $b$ and $c$ equal to zero, then
\begin{gather*}
\phi(x(s)) = -2abcx(s)^2+(ab+ac+bc+1)x(s)+abc-a-b-c, \\
\psi(x(s)) = -\frac{4abc\sqrt{q}}{ q-1}x(s)+2\frac{\sqrt{q}(ab+ac+bc-1)}{ q-1},
\\
\frac{d_{k+1}}{ d_k}=2 {\frac {{q}^{k}abc \big( {q}^{n}-{q}^{k} \big) d ( k) }{ \big( {q}^{k}q-1 \big) \big( {q}^{k}ac-1 \big) \big( {q}^{k}ab-1 \big) }},\qquad
d_k={\frac { \big( {q}^{-n},q \big)_k \big({-}2 abc{q}^{n} \big) ^{k}{q}^{{k\choose 2}}}{( q,q)_k ( ac;q )_k ( ab;q )_k }}d_0,
\end{gather*}
and the basic hypergeometric representation of the corresponding polynomial is
\begin{gather*}\sum_{k=0}^{n}\frac{\big(q^{-n};q\big)_k\big(aq^s,aq^{-s};q\big)_k}{ (ac;q)_k(ab;q)_k}\frac{\big(bcq^n\big)^k}{ (q;q)_k}=\lim_{d\rightarrow\infty}\frac{a^np_n(x;a,b,c,d\,|\,q)}{(ab;q)_n (ac;q)_n(ad;q)_n}.\end{gather*}
\item If $P(X)$ is of degree 2 and $P(X)=C(X-a)(X-b)$, $a,b\neq 0$, then
\begin{gather*}\phi(x)=2abcx^2-(a+b)x+1-ab, \qquad \psi(x)=\frac{4ab\sqrt{q}}{ q-1}x-\frac{2\sqrt{q}(a+b)}{ q-1},\\
 \frac{d_{k+1}}{ d_k}=2 {\frac{ {b} {q}^{n}\big( 1-{q}^{k}{q}^{-n} \big)
}{ \big( 1-{q}^{k}q \big) \big( 1-{q}^{k}ab \big) }
}, \qquad d_k={\frac { \big( {q}^{-n},q \big)_k ( 2 b{q}^{n}) ^{k}}{ ( q;q )_k ( ab;q )_k }}d_0,
\end{gather*}
and the basic hypergeometric representation of the corresponding polynomial is
\begin{gather*}\sum_{k=0}^{n}\frac{\big(q^{-n};q\big)_k\big(aq^s,aq^{-s};q\big)_k{q}^{-{k\choose 2}}}{ (ab;q)_k}\frac{\big(-\frac{bq^n}{a}\big)^k}{ (q;q)_k}=\lim_{c,d\rightarrow\infty}\frac{a^np_n(x;a,b,c,d\,|\,q)}{(ab;q)_n (ac;q)_n(ad;q)_n}.\end{gather*}

\item If $P(X)$ is of degree 1 and $P(X)=C(X-a)$, $a\neq 0$, then
\begin{gather*}\phi(x)=-2ax(s)^2+x(s)+a,\qquad \psi(x)=-\frac{4\sqrt{q}a}{ q-1}x(s)+\frac{2\sqrt{q}}{ q-1},\\
\frac{d_{k+1}}{ d_k}=2 {\frac { {q}^{n}\big( 1-{q}^{k}{q}^{-n} \big) }{{q}^{k
}a \big( {q}^{k}q-1 \big) }}, \qquad d_k={\frac {\big( {q}^{-n};q \big)_k {q}^{-{k\choose
2}}}{( q;q)_k } \left( -2 {\frac {{q}^{n}}{a}} \right) ^{k}}d_0,
\end{gather*}
and the basic hypergeometric representation of the corresponding polynomial is
\begin{gather*}\sum_{k=0}^{n}\big(q^{-n};q\big)_k\big(aq^s,aq^{-s};q\big)_k{q}^{-2{k\choose 2}}\frac{\big(\frac{q^n}{ a^2}\big)^k}{ (q;q)_k}=\lim_{b,c,d\rightarrow\infty}\frac{a^np_n(x;a,b,c,d\,|\,q)}{(ab;q)_n (ac;q)_n(ad;q)_n}.\end{gather*}

\item If $P(X)$ is a constant, that is $P(X)=C$, then
\begin{gather*}\phi(x)=2x(s)^2-1, \qquad \psi(x)=\frac{4\sqrt{q}}{ q-1}x(s),\end{gather*}
and the corresponding polynomial is up to a multiplicative factor equal to
\begin{gather*}
\sum_{k=0}^{[\frac{n}{ 2}]}d_{n-2k}K_{n-2k}(x(s)),
\end{gather*}
where
\begin{gather*}\frac{d_{n-2k}}{ d_{n-2(k-1)}}=\frac{1}{ 4}\frac{\big(1-q^{-n-1}q^{2k}\big)\big(1-q^{-n-2}q^{2k}\big)q^{n+2}q^{-2k}}{ 1-q^{2k}}, \\ \frac{d_{n-2k}}{d_n}=\frac{\big(q^{-n};q\big)_{2k}}{ \big(q^2;q^2\big)_k}\left(\frac{q^{n}}{ 4}\right)^{k}q^{-2{k\choose 2}}.\end{gather*}
Take $d_n=1$ to obtain
\begin{gather*}
\sum_{k=0}^{[\frac{n}{ 2}]}\frac{\big(q^{-n};q\big)_{2k}}{ \big(q^2;q^2\big)_k}\left(\frac{q^{n}}{ 4}\right)^{k}q^{-2{k\choose 2}}K_{n-2k}(x(s))=2^{-n}H_n\big(x\,|\,q^{-1}\big)\\
\qquad {} =\lim_{a,b,c,d\rightarrow\infty}\frac{2^{-n}a^{2n}p_n(x;a,b,c,d\,|\,q)}{(ab;q)_n (ac;q)_n(ad;q)_n},
\end{gather*}
where $H(x\,|\,q)$ is the continuous $q$-Hermite polynomials.
\end{enumerate}

\begin{Remark}It is important to note that in each of the cases of items 2 to 4 above, if one of the $a$, $b$, $c$, $d$ is zero, then $\psi_1=0$, which is impossible because the degree of $\psi$ is equal to~1.
\end{Remark}
In order to expand on the generalization of Bochner's theorem in \cite{VinetZhedanov}, we need to connect the second-order difference equation (\ref{vinet}) used in \cite{VinetZhedanov} to the Sturm--Liouville type equation (\ref{e1.10.1}).
Let $P_n(x(s))$ be a polynomial solution to the second-order difference equation (\ref{vinet}). From the definition of $\mathbb{D}_{x}$ and $\mathbb{S}_{x}$ we have
\begin{gather*}
\nabla x_1(s)\mathbb{D}_{x}^2 P_n(x(s))=\frac{ P_n(x(s+1))
- P_n(x(s))}{ x(s+1)-x(s)}-\frac{ P_n(x( s))- P_n(x( s-1))}{ x( s)-x( s-1)},\\
2\mathbb{S}_{x}\mathbb{D}_{x} P_n(x(s))=\frac{ P_n(x(s+1))
- P_n(x(s))}{ x(s+1)-x(s)}+\frac{ P_n(x( s))- P_n(x( s-1))}{ x( s)-x( s-1)}.
\end{gather*} Solve the system with unknowns $\{P_n(x(s+1)),P_n(x(s-1))\}$ and substitute the solution into~(\ref{vinet}) to obtain a Sturm--Liouville type equation of the form~(\ref{e1.10.1}) where the coefficient of $P_n(x(s))$ is $A( s)+B( s)+C( s)- \lambda_n$. Taking $n=0$ in~(\ref{vinet}) and using the fact that $\lambda_0=0$ (see \cite[equation~(3.1)]{VinetZhedanov}) and $P_0=1$, we obtain
$A(s)+B(s)+C(s)=0$. Therefore $P_n$ is a~polynomial solution of~(\ref{e1.10.1}). This leads us to the following restatement of the generalization of Bochner's theorem in~\cite{VinetZhedanov}:
\begin{Theorem}\label{th2a} The Sturm--Liouville type equation
\begin{gather*}
\phi(x)\mathbb{D}_{x}^2y(x)+\psi(x)\mathbb{S}_{x}\mathbb{D}_{x} y(x)+\lambda_n y(x)=0, \end{gather*}
where $\phi$ and $\psi$ are polynomials of degree at most $2$ and $1$ respectively, and~$\lambda_n$ is a constant, has a polynomial solution $P_n(x)$ of degree $n=0,1,2,3,\dots$ if and only if
\begin{enumerate}\itemsep=0pt
\item[$(1)$] On $x(s)=c_1q^{-s}+c_2q^s$, $c_1\neq0$, $P_n(x)$ is, up to a multiplicative constant, equal to \begin{gather*}
_4 \phi_3 \left(
\begin{matrix} q^{-n},u^2abcdq^{n-1},aq^{-s},auq^s\\ abu,acu,adu
\end{matrix} ; q,q \right) \\
\qquad{} =\frac{\big(u^\frac{1}{ 2}a\big)^np_n\big(\frac{u^\frac{1}{ 2}q^s+u^{-\frac{1}{ 2}}q^{-s}}{2};u^\frac{1}{2}a,u^\frac{1}{ 2}b,u^\frac{1}{ 2}c,u^\frac{1}{2}d\,|\,q\big)}{ (abu,acu,adu;q)_n},
\end{gather*}
as well as subcases including limiting cases as one or more parameters $a$, $b$, $c$, $d$, tend to~$\infty$. Here, $p_n(x;a,b,c,d\,|\,q)$ denotes Askey--Wilson~\eqref{e58} polynomials and $u=\frac{c_2}{ c_1}$.
\item[$(2)$] On $x(s)=c_4s^2+c_5s$, $c_4\neq0$, $P_n(x)$ is, up to a multiplicative constant, equal to
\begin{gather}
_4 F_3 \left(
\begin{matrix} -n,a+b+c+d+2u+n-1,a-s,a+u+s
\\ a+b+u,a+c+u,a+d+u
\end{matrix}
; 1\right) \nonumber\\
\qquad{} =\frac{W_n\big({-}\big(s+\frac{u}{ 2}\big)^2;a+\frac{u}{ 2},b+\frac{u}{ 2},c+\frac{u}{ 2},d+\frac{u}{ 2}\big)}{ (a+b+u)_n(a+c+u)_n(a+d+u)_n},\label{e48}
\end{gather}
or the polynomial
\begin{gather}\label{e49}_3 F_2 \left(
\begin{matrix} -n,a-s,a+u+s\\ a+b+u,a+c+u
\end{matrix}; 1\right) =\frac{S_n\big({-}\big(s+\frac{u}{ 2}\big)^2;a+\frac{u}{ 2},b+\frac{u}{ 2},c+\frac{u}{2}\big)}{ (a+b+u)_n(a+c+u)_n(a+d+u)_n},
\end{gather} where $u=\frac{c_5}{c_4}$, $W_n$ denotes Wilson polynomials {\rm \cite[equation~(9.1.1)]{KSL}} and $S_n$ denotes continuous dual Hahn polynomials {\rm \cite[equation~(9.3.1)]{KSL}}.
\end{enumerate}
\end{Theorem}

\begin{proof} For the proof of Theorem \ref{th2a}(1), follow the scheme described for Theorem \ref{th2} to obtain the result. For the proof of Theorem~\ref{th2a}(2):
\begin{enumerate}\itemsep=0pt
\item Take $\phi(x(s))=\phi_2x(s)^2+\phi_1x(s)+\phi_0$, $\psi(x(s)=\psi_1x(x(s))+\psi_0$ and $x(s)=c_4s^2+c_5s$ in $\sigma(x(s))$, with $X=s$, to obtain the polynomial
\begin{gather}
P(X) = {X}^{4}\phi_2{c_{{4}}}^{2}+ \big( 2 u\phi_2{c_{{4}}}^{2}-\psi_1{c_{{4}}}^{2} \big) {X}^{3}+ \big( {u}^{2}\phi_2{c_{{4}}}^{2}+\phi_{{1}}c_{{4}}-3/2 u\psi_1{c_{{4}}}^{2} \big) {X}^{2}\nonumber\\
\hphantom{P(X) =}{} + \big( u\phi_1c_{{4}}-\psi_0c_{{4}}-1/2 {u}^{2}\psi_1{c_{{4}}}^{2} \big) X+\phi_{{0}}-1/2 u\psi_0c_{{4}}, \qquad u=\frac{c_5}{ c_4}.\label{e50}
\end{gather}
\item Suppose P is of degree 4, that is $\phi_2\neq 0$:
\begin{itemize}\itemsep=0pt
\item Write
\begin{gather*}P(X)=C(X-a)(X-b)(X-c)(X-d), \qquad a,b,c,d\in \mathbb{C},\end{gather*} expand it and identify the coefficients of $X^i$, $i=0,1,2,3,4$ with those of $P$ in (\ref{e50}) to obtain a system of five equations;
\item Solve the system with unknowns $\phi_2$, $\phi_1$, $\phi_0$, $\psi_1$ and $\psi_0$ to obtain the corresponding polynomial coefficients of (\ref{e1.10.1})
\begin{gather}
\phi(x(s))=\frac{x(s)^2}{ c_4^2}\nonumber\\
\hphantom{\phi(x(s))=}{} +{\frac { {u}( 4 {u}+3a+3b+3c+3d)+2(ab+ac+ad+bc+bd+cd) }{2c_{{4}}}}x(s)\nonumber\\
\hphantom{\phi(x(s))=}{} + abcd+\frac{uabc+{u}^{2}bc+ubcd+uabd+uacd+{u}^{3}b+{u}^{3}c}{ 2\nonumber}\\
\hphantom{\phi(x(s))=}{} +\frac{{u}^{3}d+{u}^{3}a+{u}^{4}+{u}^{2}bd+{u}^{2}cd+{u}^{2}ac+{u}^{2}ab+{u}^{2}ad}{ 2},\label{e35a}\\
\psi(x(s)) = {\frac {b+c+d+a+2 u}{{c_{{4}}}^{2}}}x(s)\nonumber\\
\hphantom{\psi(x(s)) =}{}+\frac{abc+ubc+bcd+abd+acd+{u}^{2}b+{u}^{2}c+{u}^{2}d}{ c_4}\nonumber\\
\hphantom{\psi(x(s)) =}{}+\frac{{u}^{2}a+{u}^{3}+ubd+ucd+uac+uab+uad}{ c_4}.\label{e35b}
\end{gather}
\item Use Lemma \ref{Prop4} with $\eta=a$ and take into account the fact that $w_k(x(s),a)=(-c_4)^k(a+s)_k(a+u-s)_k$ on $x(s)=c_4s^2+c_5s$, (see Remark~(\ref{Prop5})), to obtain
\begin{gather*}
\sum_{k=0}^{n}\frac{(-n)_k(a+b+c+d+2u+n-1)_k}{ (a+b+u)_k(a+c+u)_k(a+d+u)_kk!}\left(\frac{-1}{ c_4}\right)^kw_k(x(s),a)\\
\qquad{} = {}_4 F_3\left(\begin{matrix} -n,a+b+c+d+2u+n-1,a-s,a+u+s\\ a+b+u,a+c+u,a+d+u
\end{matrix}
; 1 \right),\end{gather*}
where $u=\frac{c_5}{ c_4}$.
\end{itemize}
\item If $P$ is of degree 3, that is $\phi_2=0$,
\begin{gather*}
P(X)= -\psi_1{c_{{4}}}^{2} {X}^{3}+ \left(\phi_{{1}}c_{{4}}-\frac{3u\psi_1{c_{{4}}}^{2}}{2} \right) {X}^{2}\\
\hphantom{P(X)=}{} + \left( u\phi_1c_{{4}}-\phi_0c_{{4}}-\frac{{u}^{2}\psi_1{c_{{4}}}^{2}}{ 2} \right) X+\phi_{{0}}-\frac{u\psi_0c_{{4}}}{2},
\end{gather*}
write \begin{gather*}P(X)=C(X-a)(X-b)(X-c), \qquad a, b, c\in \mathbb{C},\end{gather*}
and use an algorithm analogous to the one described above to obtain the following polynomial coefficients of (\ref{e1.10.1})
\begin{gather*}
\phi(x(s))=-{\frac {( 3 u+2 a+2 b+2 c)}{2c_{{4}}}}x(s)-abc-\frac{u\big({u}^{2}+{u}a+{u}b+{u}c+ab+ac+bc\big)}{ 2},\\
\psi(x(s))=-\frac{x(s)}{ c_4^2}-{\frac {\big( {u}^{2}+ua+ub+uc+ab+ac+bc \big) }{c_{{4}}}},
\end{gather*}
and (\ref{e49}) as the corresponding polynomial system. Since $\psi_1\neq 0$, $P$ can not be of degree less than~3.\hfill \qed
\end{enumerate}\renewcommand{\qed}{}
\end{proof}

\begin{Remark}\quad
\begin{enumerate}\itemsep=0pt
\item Taking $c_4\rightarrow 1$, $c_5\rightarrow 0$ and $s\rightarrow {\rm i}s$, with ${\rm i}^2=-1$, (\ref{e48}) reads as
\begin{gather*}\frac{W_n\big(s^2;a,b,c,d\big)}{ (a+b)_n(a+c)_n(a+d)_n}=\frac{W_n(-x({\rm i}s); a,b,c,d)}{ (a+b)_n(a+c)_n(a+d)_n},\end{gather*}
where $W_n\big(s^2,a,b,c,d\big)$ is the Wilson polynomial \cite[equation~(9.1.1)]{KSL} and $x(z)=z^2$.
\item If $c_4=1$, $c_5=\gamma+\delta+1$, $a=0$, $b=\alpha-\gamma-\delta$, $c=\beta-\gamma$ and $d=-\delta$, the polynomial in~(\ref{e48}) is the Racah polynomial \cite[equation~(9.2.1)]{KSL}.
\item Taking $c_4\rightarrow 1$, $c_5\rightarrow 0$ and $s\rightarrow {\rm i}s$, with ${\rm i}^2=-1$, (\ref{e49}) reads as
\begin{gather*}
\frac{S_n\big(s^2;a,b,c\big)}{ (a+b)_n(a+c)_n}=\frac{S_n(-x({\rm i}s);a,b,c)}{ (a+b)_n(a+c)_n},
\end{gather*}
where $S_n\big(s^2;a,b,c\big)$ is the continuous dual Hahn polynomial \cite[equation~(9.3.1)]{KSL} and $x(z)=z^2$.
\end{enumerate}
\end{Remark}

\begin{Corollary}\label{co1}\quad\samepage
\begin{enumerate}\itemsep=0pt
\item[$(1)$] Wilson polynomials satisfy the Sturm--Liouville type equation~\eqref{e1.10.1} with
\begin{gather*}
\phi(x(z))=x(z)^2+(ab+ac+ad+cd+cb+bd)x(z)+abcd, \\
\psi(x(z))=(a+b+c+d)x(z)+abc+abd+acd+bcd, \\
\lambda_n=-n(a+b+c+d+n-1).
\end{gather*}
\item[$(2)$] Continuous dual Hahn polynomials satisfy the Sturm--Liouville type equation~\eqref{e1.10.1} with
\begin{gather*}
\phi(x(z))=(a+b+c)x(z)+abc, \qquad \psi(x(z))=x(z)+ab+ac+bc, \qquad \lambda_n=-n.
\end{gather*}
In both cases, $x(z)=z^2$ $(z={\rm i}s$, ${\rm i}^2=-1)$.
\end{enumerate}
\end{Corollary}
\begin{proof}$(1)$ Take $c_4\rightarrow 1$, $c_5\rightarrow 0$ and $s\rightarrow z$, with $z={\rm i}s$, ${\rm i}^2=-1$ in (\ref{e35a}) and (\ref{e35b}) to obtain polynomial coefficients of (\ref{e1.10.1}), $x(z)=z^2$, then use (\ref{e39}) to get $\lambda_n=-n(a+b+c+d+n-1)$. $(2)$ is obtained in a similar way.
\end{proof}

\begin{Corollary}\label{cor1}The Sturm--Liouville type equation \eqref{e1.10.1} has a polynomial solution $P_n(x)$ of degree $n=1,2,3,\dots$ if and only if $P_n(x)$ is a multiple of a Wilson polynomial, continuous dual Hahn polynomial or Askey--Wilson polynomial $p_n(x;a,b,c,d\,|\,q)$ and subcases, including limiting cases as one or more parameters $a$, $b$, $c$, $d$, tend to $\infty$.
\end{Corollary}
\begin{Remark} In~\cite{niki1991}, Nikiforov et al.\ classified orthogonal polynomials of the quadratic and $q$-quadratic variable by solving the Pearson type equation \cite[equation~(3.2.9)]{niki1991} and obtained Racah polynomials, dual Hahn polynomials and their $q$-analogs. Our approach, which is based on the Sturm--Liouville type equation (see Theorem \ref{th2a}) and uses the polynomial $P$ (appearing in the proof of Theorems~\ref{th2} and~\ref{th2a}), Lemmas \ref{Prop4} and~\ref{Prop6}, and Remark~\ref{Prop5}, leads, in addition to Racah polynomials and dual Hahn polynomials, to Askey--Wilson polynomials, subcases and limiting cases. This completes the result in Nikiforov et al.~\cite{niki1991}, generalizes \cite[Theorem~3.1]{Mourad2003} and provides explicit solutions to \cite[equation~(1.3)]{VinetZhedanov}. To the best of our knowledge, our treatment of the generalized Bochner Theorem is new.
\end{Remark}

\section[Structure relations of orthogonal polynomials of the quadratic and $q$-quadratic variable]{Structure relations of orthogonal polynomials of the quadratic\\ and $\boldsymbol{q}$-quadratic variable}\label{3}
In this section, for a family $\{P_n(x(s))\}_{n\geq 0}$ of classical orthogonal polynomials of the quadratic and $q$-quadratic variable, we prove equivalence between the Sturm--Liouville type equation~(\ref{e1.10.1}), the orthogonality of the second derivatives $\big\{\mathbb{D}_{x}^2P_n\big\}_{n\geq 2}$ and a first structure relation that genera\-li\-zes~(\ref{structureAW}). This will enable us to derive, from Theorem~\ref{th2a}, the solution to the Askey problem related to~(\ref{vinet}) and a second structure relation for classical orthogonal polynomials of the quadratic and $q$-quadratic variable.

We begin by generalizing \cite[Lemma 3.1]{MK2017}
\begin{Lemma}\label{Prop1} Let $\{P_n\}_{n=0}^{\infty}$ be a sequence of monic orthogonal polynomials. If there exist two sequen\-ces $(a'_n)$ and $(b'_n)$ of numbers such that
\begin{gather*}
\frac{1}{ \gamma_{n+1}}\mathbb{D}_{x} P_{n+1}(x)=(x-a'_n)\frac{1}{ \gamma_{n}}\mathbb{D}_{x} P_{n}(x)-\frac{b'_n}{ \gamma_{n-1}}\mathbb{D}_{x} P_{n-1}(x)+c_n, \qquad c_n \in \mathbb{C},
\end{gather*}
then, there exist two polynomials $\phi(x)$ and $\psi(x)$ of degree at most two and of degree one respectively, and $\lambda_n$ a constant such that $P_n(x)$ satisfies the divided-difference equation
\begin{gather}\label{e18a}
\phi(x) \mathbb{D}_{x}^2P_n(x)+\psi(x)\mathbb{S}_{x}\mathbb{D}_{x} P_n(x)+\lambda_n P_n(x)=0,\qquad n\geq 5.
\end{gather}
\end{Lemma}
\begin{proof} The proof follows exactly the same argument of the proof of \cite[Lemma~3.1]{MK2017}, replacing the averaging operator $\mathcal{S}_q$ by $\mathbb{S}_{x}$, the Askey--Wilson operator $\mathcal{D}_q$ by $\mathbb{D}_{x}$ and where the polynomials~$U_1$ and~$U_2$ are those appearing in~(\ref{e21a}) and~(\ref{e16a}).
\end{proof}

\begin{Theorem}\label{th1} Let $\{P_n\}_{n=0}^{\infty}$ be a sequence of polynomials orthogonal with respect to a positive weight function $w(x)$. The following properties are equivalent:
\begin{enumerate}\itemsep=0pt
\item[$a)$] There exists a polynomial $\pi(x)$ of degree at most $4$ and sequences of five elements \linebreak $\{a_{n,n+k}\}_{n\geq 2}$, $-2\leq k\leq 2$, $a_{n,n-2}\neq 0$ such that $P_n$ satisfies the structure relation
\begin{gather}\label{e9}
\pi(x)\mathbb{D}_{x}^2P_n(x)=\sum_{k=-2}^{2}a_{n,n+k}P_{n+k}(x).
\end{gather}
\item[$b)$] There exists a polynomial $\pi(x)$ of degree at most four such that $\big\{\mathbb{D}_{x}^2P_n\big\}_{n=2}^{\infty}$ is orthogonal with respect to $\pi(x) w(x)$.
\item[$c)$] There exist two polynomials $\phi(x)$ and $\psi(x)$ of degree at most two and of degree one respectively, and a constant $\lambda_n$ such that
\begin{gather}\label{e10d}
\phi(x) \mathbb{D}_{x}^2P_n(x)+\psi(x)\mathbb{S}_{x}\mathbb{D}_{x} P_n(x)+\lambda_n P_n(x)=0, \qquad n\geq 5.
\end{gather}
\end{enumerate}
\end{Theorem}
\begin{proof}The proof is organized as follows:

Step 1. $(a)\Rightarrow (b)\Rightarrow (a)$ which is equivalent to $(a)\Leftrightarrow (b)$.

Step 2. $(b)\Rightarrow (c)\Rightarrow (a)$ which, taking into account Step 1, is equivalent to $(b)\Leftrightarrow (c)$.

Step 1: We assume that $(a)$ is satisfied and we prove $(b)$. Let $m\geq 2$ and $n\geq 2$ be two integers, and assume that $m\leq n$. From $(a)$, there exists a polynomial $\pi$ of degree at most four and there exist sequences of five elements $\{a_{n,n+j}\}_n$, $j=-2, -1, 0, 1, 2$ such that
\begin{gather}\label{e30}
\pi\mathbb{D}_{x}^2P_n=\sum_{j=-2}^{2}a_{n,n+j}P_{n+j}, \qquad \text{with} \quad a_{n,n-2}\neq 0.
\end{gather}
Since $m\leq n$, $m-2\leq n-2\leq n+j\leq n+2$ (for $j=-2, -1, 0, 1, 2$). So, multiplying both sides of (\ref{e30}) by $W\mathbb{D}_{x}^2P_m$, integrating on $(a, b)$ and then taking into account the fact that $(P_n)$ is orthogonal on the interval $(a, b)$ with respect to the weight function $W$, we obtain
\begin{gather*}
\int_{a}^{b}\mathbb{D}_{x}^2P_m(x)\mathbb{D}_{x}^2P_n(x)\pi (x) W(x){\rm d}x \
 \begin{cases}
=0& \text{if} \ m<n,\\
\neq 0& \text{if} \ m=n.
\end{cases}
\end{gather*}
If $n<m$, we substitute in (\ref{e30}), $n$ by $m$ and by a similar way, we obtain
\begin{gather*}\int_{a}^{b}\pi (x) W(x)\mathbb{D}_{x}^2P_n(x)\mathbb{D}_{x}^2P_m(x){\rm d}x=0.\end{gather*}

Now we assume $(b)$ and we prove $(a)$. Since $\pi (x)\mathbb{D}_{x}^2P_n$ is a polynomial of degree less or equal to $n+2$, $\pi(x)\mathbb{D}_{x}^2P_n$ can be expanded in the orthogonal basis $\{P_j\}_{j=0}^{\infty}$ as
\begin{gather*}\pi (x)\mathbb{D}_{x}^2P_n=\sum_{j=0}^{n+2}a_{n,j}P_j,\end{gather*}
where $a_{n,j}$, $j=0,\dots,n+2$ is given by
\begin{gather*} a_{n,j}\int_{a}^{b}W(x)\big(\mathbb{D}_{x}^2P_n(x)\big)^2{\rm d}x=\int_{a}^{b}\pi(x) W(x)P_j(x)\mathbb{D}_{x}^2P_n(x){\rm d}x.\end{gather*}
Since $\mathbb{D}_{x}^2P_n(x)$ is of degree $n-2$ we deduce from the hypothesis that $a_{n,j}=0$ for $j=0,\dots,n-2$ and $a_{n,n-2}\neq 0$.

Step 2: We suppose $(b)$ and we prove $(c)$. Firstly, we prove that $\{P_n\}_{n=0}^{\infty}$ satisfies an equation of type~(\ref{e18a}). Since $\frac{x}{ \gamma_n}\mathbb{D}_{x} P_n$ is a monic polynomial of degree~$n$, it can be expanded as
\begin{gather}\label{e31a}
x\frac{1}{ \gamma_n}\mathbb{D}_{x} P_n=\frac{1}{ \gamma_{n+1}}\mathbb{D}_{x} P_{n+1}+\sum_{j=1}^{n}\frac{e_{n,j}}{ \gamma_j}\mathbb{D}_{x} P_j, \qquad e_{n,j}\in \mathbb{R}.
\end{gather}
Applying $\mathbb{D}_{x}$ to both sides, we obtain
\begin{gather}
(\alpha x+\beta)\frac{1}{ \gamma_n}\mathbb{D}_{x}^2 P_n+\frac{1}{ \gamma_n}\mathbb{S}_{x}\mathbb{D}_{x} P_n=\frac{1}{ \gamma_{n+1}}\mathbb{D}_{x}^2 P_{n+1}+\sum_{j=2}^{n}\frac{e_{n,j}}{ \gamma_j}\mathbb{D}_{x}^2P_j.\label{e31}
\end{gather}
Apply $\mathbb{D}_{x}$ to both sides of the three-term recurrence relation (\ref{e63}), use the product rule (\ref{e16}) and the fact $\mathbb{S}_{x} (x(s))=\alpha x(s)+\beta$ to obtain
\begin{gather}\label{e31aa}
\mathbb{D}_{x} P_{n+1}=(\alpha x+\beta-a_n)\mathbb{D}_{x} P_n+\mathbb{S}_{x} P_n-b_n\mathbb{D}_{x} P_{n-1}.
\end{gather}
Apply $\mathbb{D}_{x}$ to both sides and use (\ref{e21a}) as well as the expression of $U_1$, $U_1(x)=\big(\alpha^2-1\big)x+\beta (\alpha +1)$ to obtain
\begin{gather}\label{e32a}
\mathbb{D}_{x}^2P_{n+1}=\big[\big(2\alpha^2-1\big)x+2\beta (\alpha+1 )-a_n]\mathbb{D}_{x}^2P_n+2\alpha\mathbb{S}_{x}\mathbb{D}_{x}P_n-b_n\mathbb{D}_{x}^2P_{n-1}.
\end{gather}
By using this relation to eliminate $\mathbb{S}_{x}\mathbb{D}_{x} P_n$ in (\ref{e31}) we obtain
\begin{gather}
\frac{1}{ \gamma_n}x\mathbb{D}_{x}^2 P_n-\frac{\left(2\beta+a_n\right)}{ \gamma_n}\mathbb{D}_{x}^2P_n+\frac{b_n}{ \gamma_n}\mathbb{D}_{x}^2P_{n-1}\nonumber\\
\qquad{} =\left(\frac{2\alpha}{ \gamma_{n+1}}-\frac{1}{ \gamma_n}\right)\mathbb{D}_{x}^2 P_{n+1}+\sum_{j=2}^{n}\frac{2\alpha e_{n,j}}{ \gamma_j}\mathbb{D}_{x}^2P_j.\label{e32}
\end{gather}
Since $\big\{\frac{\mathbb{D}_{x}^2 P_n}{ \gamma_{n}\gamma_{n-1}}\big\}$ is orthogonal, there exist $a''_n$ and $b''_n$ such that
\begin{gather}\label{e34}
x\frac{\mathbb{D}_{x}^2P_n}{ \gamma_n}=\frac{\gamma_{n-1}}{ \gamma_{n+1}\gamma_{n}}\mathbb{D}_{x}^2P_{n+1}+a''_n\mathbb{D}_{x}^2P_n+b''_n\mathbb{D}_{x}^2P_{n-1}.
\end{gather}
So, using the relation $\gamma_{n+1}-2\alpha\gamma_n+\gamma_{n-1}=0$, obtained by direct computation, (\ref{e32}) becomes
\begin{gather*}\left(a''_n-\frac{2\beta+a_n}{ \gamma_n}\right)\mathbb{D}_{x}^2P_n+\left(b''_n+\frac{b_n}{ \gamma_n}\right)\mathbb{D}_{x}^2P_{n-1}=\sum_{j=2}^{n}\frac{2\alpha e_{n,j}}{ \gamma_j}\mathbb{D}_{x}^2P_j.\end{gather*}
Therefore,
$e_{n,j}=0$ (for $j=2,3,.\dots,n-2$) and (\ref{e31a}) reads as
\begin{gather*}
\frac{x}{ \gamma_n}\mathbb{D}_{x} P_n=\frac{1}{ \gamma_{n+1}}\mathbb{D}_{x} P_{n+1}+\frac{e_{n,n}}{ \gamma_n}\mathbb{D}_{x} P_n+\frac{e_{n,n-1}}{ \gamma_{n-1}}\mathbb{D}_{x} P_{n-1}+e_{n,1}.
\end{gather*}
Then, from Lemma \ref{Prop1}, there exist two polynomials $\phi$ of degree at most 2 and $\psi$ of degree 1, and a constant $\lambda_n$ such that $P_n$ satisfies
\begin{gather*}
\phi\mathbb{D}_{x}^2P_n+\psi \mathbb{S}_{x}\mathbb{D}_{x} P_n+ \lambda_nP_n=0,\qquad n\geq 5. \end{gather*}

Let us prove that $(c)\Rightarrow (a)$. Note that $P_n$, $n=1,2,3,4$, satisfies (\ref{e10d}). In fact, it follows from the algorithm described in the proof of Theorem \ref{th2a}, that, for ${n\geq 5}$, $P_n$ is up to a multiplicative factor equal to Askey--Wilson polynomial and Wilson polynomial, and subcases, including limiting cases, denoted by $p_n$. Since $\{p_n\}_{n=0}^{\infty}$ is orthogonal (cf.~\cite{MK2017, KSL}) both families are orthogonal with respect to the same measure. Therefore $P_n$, $n=1,2,3,4$, is up to a multiplicative factor equal to $p_n$ which satisfies (\ref{e10d}) by Theorem~\ref{th2a}).

Apply $\mathbb{S}_{x}$ to both sides of (\ref{e31aa}) and use (\ref{e16}) as well as (\ref{e16b}) to obtain
\begin{gather*}\mathbb{S}_{x}\mathbb{D}_{x} P_{n+1}=\big(\alpha^2 x+U_1(x)+\beta(\alpha+1)-a_n\big)\mathbb{S}_{x}\mathbb{D}_{x} P_n+2\alpha U_2\mathbb{D}_{x}^2P_n+P_n-b_n\mathbb{S}_{x}\mathbb{D}_{x} P_{n-1}.\end{gather*} Adding $\psi$ times the previous equation and $\phi$ times (\ref{e32a}), and then using the assumption~$(c)$, we obtain
\begin{gather*}\lambda_{n+1}P_{n+1}=\lambda_n\big(\alpha^2x+U_1(x)+\beta (\alpha+1 )-a_n\big)P_n\\
\hphantom{\lambda_{n+1}P_{n+1}=}{} -2\alpha\big(\phi\mathbb{S}_{x}\mathbb{D}_{x}P_n+U_2\psi\mathbb{D}_{x}^2P_n\big)
-\psi P_n-b_n\lambda_{n-1}P_{n-1}.\end{gather*}
Multiplying the latter equation by $\psi$ and using the relation $\psi\mathbb{S}_{x}\mathbb{D}_{x} P_n=-\phi\mathbb{D}_{x}^2P_n-\lambda_nP_n$, obtained from the assumption, and then substituting $U_1$ by $\big(\alpha^2-1\big)x+\beta (\alpha+1)$, we obtain
\begin{gather*} 2\alpha\big(\phi^2-U_2\psi^2\big)\mathbb{D}_{x}^2P_n =\lambda_{n+1}\psi P_{n+1}\\
\qquad{} +\big[\psi^2-2\alpha\lambda_n\phi-\lambda_n\big(\big(2\alpha^2-1\big)x
+2\beta(\alpha+1)-a_n\big)\psi\big]P_n+\lambda_{n-1}b_n\psi P_{n-1}.\end{gather*}
Taking $\phi (x)=\phi_2x^2+\phi_1x+\phi_0$ and $\psi (x)=\psi_1x+\psi_0$ and using the three-term recurrence rela\-tion~(\ref{e34}), we transform the above equation into
\begin{gather}\label{e34a}
\big(\phi^2-U_2\psi^2\big)\mathbb{D}_{x}^2P_n=\sum_{j=-2}^{2}a_{n,n+j}P_{n+j},
\end{gather}
where
\begin{gather*} a_{n,n-2}=\frac{\big[\psi_1^2-\lambda_n\big(2\alpha\phi_2+\big(2\alpha^2-1\big)\psi_1\big)\big]b_{n}b_{n-1}+\psi_1b_{n-1}b_n\lambda_{n-1}}{ 2\alpha}.
\end{gather*} $a_{n,n-2}\neq 0$, for $b_n>0$, $n=0,1,\dots$, and $\psi_1$ does not depend on $n$.
\end{proof}

\begin{Corollary}A family of monic orthogonal polynomials $\{P_n\}_{n=0}^{\infty}$ satisfies the relation~\eqref{e9} if and only if $P_n(x)$ is a multiple of the Wilson polynomial, continuous dual Hahn polynomial or Askey--Wilson polynomial $p_n(x;a,b,c,d\,|\,q)$ and subcases, including limiting cases as one or more parameters $a$, $b$, $c$, $d$ tend to infinity.
\end{Corollary}
\begin{proof} The proof is deduced from Theorem \ref{th1}, Corollary~\ref{cor1} and the fact that limiting cases of Askey--Wilson polynomials as one or more parameters $a$, $b$, $c$, $d$ tend to $\infty$ are orthogonal polynomials families (cf.~\cite[Remark~3.2]{Mourad2003}), see also~\cite{MK2017}.
\end{proof}

Next, we turn our attention to the second structure relation.

\begin{Proposition}\label{Prop2}Let $\{P_n\}$ be a sequence of polynomials orthogonal with respect to a weight function $W$ on $(a, b)$. If $\big\{\mathbb{D}_{x}^2P_n\big\}$ is orthogonal with respect to the weight function~$\pi W$ where~$\pi$ is a polynomial of degree at most~$4$, then there exist sequences of five elements $\{b_{n,n+j}\}$, $j=-2, -1, 0, 1, 2$ such that
\begin{gather*}
P_n=\sum_{j=-2}^{2}b_{n,n+j}\mathbb{D}_{x}^2P_{n+j},\qquad b_{n,n+2}\neq0,\qquad n=2,3,\dots.
\end{gather*}
\end{Proposition}
\begin{proof} Replace $\mathcal{D}_q$ by $\mathbb{D}_{x}$ in the proof of \cite[Proposition 3.8]{MK2017}.
\end{proof}
\begin{Corollary}\label{co2} Wilson polynomials, continuous dual Hahn polynomials, Askey--Wilson polynomials, special cases and limiting cases as one or more parameters $a$, $b$, $c$, $d$ tend to $\infty$, satisfy the structure relation
\begin{gather}\label{secondSR}
P_n(x)=\sum_{j=-2}^{2}b_{n,n+j}\mathbb{D}_{x}^2P_{n+j}(x).
\end{gather}
\end{Corollary}
\begin{proof}Since those polynomials are orthogonal and they are the only solutions to (\ref{e1.10.1}), the result is obtained by using Theorem~\ref{th1} and Proposition~\ref{Prop2}.
\end{proof}

In the following proposition we show the connection of the structure relation (cf.~\cite[p.~118]{CSM}) given by~(\ref{macellan}) to~(\ref{secondSR}).
\begin{Proposition}The structure relation \eqref{macellan} is connected to our second structure rela\-tion~\eqref{secondSR} as follows:
 \begin{gather*}
\mathbb{S}_{x}\mathcal{M}P_n(x(s))=P_n(x(s))+
\big(2\alpha\big)^{-1}U_1\mathbb{D}_{x}^2P_{n+1}(x(s))+\big(2\alpha\big)^{-1}U_1b_n\mathbb{D}_{x}^2P_{n-1}(x(s))\\
\hphantom{\mathbb{S}_{x}\mathcal{M}P_n(x(s))=}{} +\big(2\alpha\big)^{-1}\big(U_1\big(a_n-\alpha^2x(s)-\beta(\alpha+1)-U_1\big)+2\alpha^2U_2\big)\mathbb{D}_{x}^2P_n(x(s)),\end{gather*}
where $a_n$ and $b_n$ are coefficients of the three-term recurrence relation~\eqref{e63}.
\end{Proposition}
\begin{proof}Observe that $\mathbb{S}_{x}\mathcal{M}P_n(x(s))=\mathbb{S}_{x}^2P_n(x(s))$, then take into account~(\ref{e16b}) to obtain
\begin{gather*}\mathbb{S}_{x}^2{P_n}(x(s))=P_n(x(s))+U_1(x(s))\mathbb{S}_{x}\mathbb{D}_{x} P_n(x(s))+\alpha U_2(x(s))\mathbb{D}_{x}^2P_n(x(s)).\end{gather*}
Then use (\ref{e32a}) to eliminate $\mathbb{S}_{x}\mathbb{D}_{x}$. \end{proof}

\begin{Remark}From (\ref{e34a}), $\big\{\mathbb{D}_{x}^2P_n\big\}_{n\geq 2}$ is orthogonal with respect to \begin{gather*}\big(\phi^2(x(s))-U_2(x(s))\psi^2(x(s))\big)W(x(s)).\end{gather*} So, there exists a constant $c>0$ such that
\begin{gather*}
\pi(x(s)) = c\big(\phi^2(x(s))-U_2(x(s))\psi^2(x(s))\big) \\
\hphantom{\pi(x(s))}{} = c\left(\phi(x(s))-\frac{\nabla x_1(s)}{ 2}\psi(x(s))\right)\left(\phi(x)+\frac{\nabla x_1(s)}{ 2}\psi(x(s))\right) \\
\hphantom{\pi(x(s))}{} \qquad \text{for} \quad U_2(x(s))=\left(\frac{\nabla x_1(s)}{ 2}\right)^2 \\
\hphantom{\pi(x(s))}{} = c\sigma(x(s))\tau(x(s)),
\end{gather*}
where $\sigma(x(s))$ and $\tau(x(s))$ are functions defined by
\begin{gather*}
\sigma(x(s)) = \phi(x(s))-\frac{\nabla x_1(s)}{ 2}\psi(x(s)),\qquad
\tau(x(s)) = \phi(x(s))+\frac{\nabla x_1(s)}{ 2}\psi(x(s)).\end{gather*}
So, without loss of generality, we can take
\begin{gather*}\pi(x(s))= \phi^2(x(s))-U_2(x(s))\psi^2(x(s)).\end{gather*}
Let us mention that the function $\sigma(x(s))$ is the one defined by~(\ref{sigma}) and also appearing in the proof of Theorem~\ref{th2} and that of Theorem~\ref{th2a}.
\end{Remark}

\section{Coefficients of the structure relations}\label{4}
\subsection{The Wilson polynomials}
\begin{Proposition} The first structure relation \eqref{e9} for monic Wilson polynomials $P_n\big(s^2;a,b,$ $c,d\big)$ is
\begin{gather*}
\big(s^2+a^2\big)\big(s^2+b^2\big)\big(s^2+c^2\big)\big(s^2+d^2\big)\frac{\delta^2P_n(s^2;a,b,c,d)}{ \delta^2s^2}=\sum_{j=-2}^{2}a_{n,n+j}P_{n+j}\big(s^2;a,b,c,d\big),
\end{gather*}
where
\begin{gather*}
a_{n,n+2}=n(n-1),\\
\frac{a_{n,n+1}}{ n(n-1)}= A_n(c,b,a,d+1) +A_{n-1}( b,a,c+1,d+1)\\
\hphantom{\frac{a_{n,n+1}}{ n(n-1)}=}{} +A_{n-2}(a,b+1,c+1,d+1) +A_{n+1}(d,b,c,a),\\
\frac{a_{n,n}}{ n(n-1)} = \big(A_n(c,b,a,d+1 ) +A_{n-1} (b,a,c +1,d+1 ) \\
\hphantom{\frac{a_{n,n}}{ n(n-1)} =}{} +A_{n-2}(a,b+1,c+1,d+1)\big)A_n(d,b,c,a ) + \big(A_{n-1} (b,a,c+1,d+1 ) \\
\hphantom{\frac{a_{n,n}}{ n(n-1)} =}{} +A_{n-2} (a,b+1,c+1,d+1)\big) A_{n-1} (c,b,a,d+1 )\\
\hphantom{\frac{a_{n,n}}{ n(n-1)} =}{} +A_{n-2} (a,b+1,c+1,d+1) A_{n-2} (b,a,c+1,d+1 ),\\
\frac{a_{n,n-1}}{ n(n-1)}=\big[ \big( A_{n-1} (b,a,c+1,d+1 ) +A_{n-2}(a,b+1,c+1,d+1 ) \big)A_{n-1} (c,b,a,d+1 )\\
\hphantom{\frac{a_{n,n-1}}{ n(n-1)}=}{} + A_{n-2} ( a,b+1,c+1,d+1 ) A_{n-2} (b,a,c+1,d+1 ) \big] A_{n-1} (d,b,c,a )\\
\hphantom{\frac{a_{n,n-1}}{ n(n-1)}=}{} + A_{n-2} (a,b+1,c+1,d+1 ) A_{n-2} (b,a,c+1,d+1 ) A_{n-2} (c,b,a,d+1 ),\\
\frac{a_{n,n-2}}{ n(n-1)} = A_{n-2} ( a,b+1,c+1,d+1 ) A_{n-2} ( b,a,c+1,d+1 )\\
\hphantom{\frac{a_{n,n-2}}{ n(n-1)} =}{}\times A_{n-2} ( c,b,a,d+1 ) A_{n-2} (d,b,c,a ),
\end{gather*}
and $A_n(a,b,c,d)$ is defined in Lemma~{\rm \ref{lemma1}}.
\end{Proposition}
\begin{proof}Substitute the polynomial coefficients $\phi(x({\rm i}s))$ and $\psi(x({\rm i}s))$, $x(z)=z^2$ of (\ref{e10d}) given in Corollary~\ref{co1} and take into account the fact that
\begin{gather*}
U_2(x({\rm i}s))=\left(\frac{x\big({\rm i}s+\frac{1}{ 2}\big)-x\big({\rm i}s-\frac{1}{ 2}\big)}{ 2}\right)^2
\end{gather*} to obtain
$\pi(x({\rm i}s))=\big(s^2+a^2\big)\big(s^2+b^2\big)\big(s^2+c^2\big)\big(s^2+d^2\big)$. Since
\begin{gather}\label{e10m}
\mathbb{D}_{x}^2 P_n\big(s^2;a,b,c,d\big)=n(n-1)P_{n-2}\big(s^2;a+1,b+1,c+1,d+1\big)=\frac{\delta^2P_n\big(s^2;a,b,c,d\big)}{ \delta^2s}
\end{gather} (see (\ref{nderiv})), the structure relation (\ref{e9}) becomes
\begin{gather*}
\big(s^2+a^2\big)\big(s^2+b^2\big)\big(s^2+c^2\big)\big(s^2+d^2\big)P_{n-2}\big(s^2;a+1,b+1,c+1,d+1\big)\\
\qquad{}=\sum_{j=-2}^{2}\frac{a_{n,n+j}}{ n(n-1)}P_{n+j}\big(s^2;a,b,c,d\big).
\end{gather*}
Replace $n$ by $n-2$ is the first equation of Lemma~\ref{lemma1}, multiply the obtained equation by $\big(s^2+b^2\big)\big(s^2+c^2\big)\big(s^2+d^2\big)$ and use the second, third and the fourth relation of this lemma to get $a_{n,n+j}$, $j\in\{-2,\dots,2\}$ in terms of~$A_{n+j}$.
\end{proof}

\begin{Proposition} The second structure relation \eqref{secondSR} for monic Wilson polynomials $P_n\big(s^2;a,b,$ $c,d\big)$ is
\begin{gather*}
P_n\big(s^2;a,b,c,d\big)=\sum_{j=-2}^{2}{b_{n,n+j}}\frac{\delta^2P_{n+j}\big(s^2;a,b,c,d\big)}{ \delta^2s^2},
\end{gather*}
	where
\begin{gather*}
(n+2)(n+1)b_{n,n+2}=1,\\
(n+1)(n)b_{n,n+1}= C_n(b,a+1,c,d) + C_n(a,b,c,d) + C_n(d,c+1,a+1,b+1)\\
\hphantom{(n+1)(n)b_{n,n+1}=}{} +C_n(c,d,a+1,b+1),\\
(n)(n-1)b_{n,n}= \big( C_n(b,a+1,c,d) +C_n(a,b,c,d)\big) \big( C_{n-1}(d,c+1,a+1,b+1)\\
\hphantom{(n)(n-1)b_{n,n}=}{} + C_{n-1}(c,d,a+1,b+1) \big) +C_n(c,d,a+1,b+1)\\
\hphantom{(n)(n-1)b_{n,n}=}{}\times C_{n-1}(d,c+1,a+1,b+1) +C_n(a,b,c,d) C_{n-1}(b,a+1,c,d),\\
(n-1)(n-2)b_{n,n-1}=\big( C_n(b,a+1,c,d) +C_n(a,b,c,d)\big) C_{n-1}(c,d,a+1,b+1)\\
\hphantom{(n-1)(n-2)b_{n,n-1}}{} \times C_{n-2}(d,c+1,a+1,b+1) +C_n(a,b,c,d) C_{n-1}(b,a+1,c,d)\\
\hphantom{(n-1)(n-2)b_{n,n-1}}{} \times \big( C_n( n-2,d,c+1,a+1,b+1) +C_{n-2}(c,d,a+1,b+1)\big),\\
(n-2)(n-3)b_{n,n-2}=C_n(b,a,c,d) C_{n-1}(a,b+1,c,d) C_{n-2}(d,c,a+1,b+1) \\
\hphantom{(n-2)(n-3)b_{n,n-2}=}{} \times C_{n-3}(c,d+1,a+1,b+1),\end{gather*}
and $C_n(a,b,c,d)$ is given in Lemma~{\rm \ref{lemma1}}.
\end{Proposition}

\begin{proof} From Corollary \ref{co2} and the relation (\ref{e10m}), we obtain
\begin{gather*}
P_n\big(s^2;a,b,c,d\big)=\sum_{j=-2}^{2}(n+j)(n+j-1)b_{n,n+j}P_{n-2+j}\big(s^2;a+1,b+1,c+1,d+1\big).
\end{gather*}
Take into account (\ref{e10l}) to obtain
\begin{gather}
P_n\big(s^2;a,b,c,d\big)= P_n\big(s^2;a+1,b+1,c,d\big) + \big( C_n (b,a+1,c,d ) +C_n (a,b,c,d ) \big)\nonumber\\
\hphantom{P_n\big(s^2;a,b,c,d\big)=}{}\times P_{n-1} \big(s^2;a+1,b+1,c,d \big)\nonumber\\
\hphantom{P_n\big(s^2;a,b,c,d\big)=}{}+C_n (a,b,c,d ) C_{n-1} (b,a+1,c,d ) P_{n-2} (a+1,b+1,c,d ). \label{e10n}
\end{gather}

Substitute $c$ by $c+1$ and $d$ by $d+1$ to obtain
\begin{gather*}
	P_n\big(s^2;a,b,c+1,d+1\big)=P_n \big(s^2;a+1,b+1,c+1,d+1 \big)\\
	\qquad{} + \big( C_n (b,a+1,c+1,d+1 ) +C_n (a,b,c+1,d+1 ) P_{n-1} \big(s^2;a+1,b+1,c+1,d+1 \big)\\
	\qquad{} +C_n (a,b,c+1,d+1 ) C_{n-1} (b,a+1,c+1,d+1 )  P_{n-2} \big( s^2;a+1,b+1,c+1,d+1 \big).
\end{gather*}
Permute $a$ and $c$, $b$ and $d$ and use the fact that $P_n$ is symmetric with respect to its parameters to obtain
\begin{gather*}
P_n\big(s^2;a,b,c+1,d+1\big)=P_n \big(s^2;d+1,c+1,a+1,b+1 \big)\\
\qquad{}+ \big( C_n (d,c+1,a+1,b+1) +C_n (c,d,a+1,b+1 ) \big) P_{n-1}
\big(s^2;d+1,c+1,a+1,b+1 \big)\\
\qquad{} +C_n (c,d,a+1,b +1 ) C_{n-1} (d,c+1,a+1,b+1 ) P_{n-2} \big(s^2;d+1,c+1,a+1,b+1 \big).
\end{gather*}
Take this relation into account in (\ref{e10n}) to obtain the result.
\end{proof}
\subsection{The Askey--Wilson polynomials}
Coefficients of the first structure relation (\ref{e9}) for the Askey--Wilson polynomials have been given in~\cite{MK2017}. As for those of the second structure relation we have the following:

\begin{Proposition} The Askey--Wilson polynomials satisfy the second structure relation
\begin{gather*}P_n(x;a,b,c,d\,|\,q)=\sum_{j=-2}^{2}b_{n,n+j}\mathcal{D}_q^2P_{n+j}(x;a,b,c,d\,|\,q), \qquad x=\cos\theta,\end{gather*}
where
\begin{gather*}
\gamma_{n+2}\gamma_{n+1}b_{n,n+2}=1,\\
-2\gamma_{n}\gamma_{n-1}b_{n,n+1}=C_n (b,aq,c,d ) +C_n (a,b,c,d )+C_n (d,cq,aq,bq ) +C_n (c,d,aq,bq),\\
4\gamma_{n+1}\gamma_{n}b_{n,n}= ( C_n (b,aq,c,d ) +C_n (a,b,c,d) ) ( C_{n-1} (d,cq,aq,bq ) +C_{n-1} (c,d,aq,bq ) )\\
\qquad{} +C_n (c,d,a q,bq ) C_{n-1} (d,cq,aq,bq ) +C_n (a,b,c,d ) C_{n-1} (b,aq,c,d ),\\
-8\gamma_{n}\gamma_{n-1}b_{n,n-1}= ( C_n (b,aq,c,d ) +C_n (a,b,c,d) ) C_{n-1} (c,d,aq,bq ) C_{n-2}(d,cq,aq,bq )\\
 \qquad {}+C_n (a,b,c,d ) C_{n-1} (b,aq,c,d ) ( C_{n-2} (d,cq,a q,bq ) +C_{n-2} (c,d,aq,bq ) ),\\
16\gamma_{n-1}\gamma_{n-2}b_{n,n-2}=C_n ( a,b,c,d ) C_{n-1} (b,aq,c,d) C_{n-2} (c,d,aq,bq ) C_{n-3} (d,cq,aq,bq ),
\end{gather*}
and \begin{gather*} C_n(a,b,c,d)=\frac {a \big( 1-{q}^{n} \big) \big( 1-bc{q}^{n-1} \big)
\big( 1-bd{q}^{n-1} \big) \big( 1-dc{q}^{n-1} \big) }{ \big(
1-abcd{q}^{2 n-2} \big) \big( 1-abcd{q}^{2 n-1} \big) }
\end{gather*} is the coefficient appearing in Lemma~{\rm \ref{lemma2}}.
\end{Proposition}
\begin{proof} Use Corollary \ref{co2}, the relation (\ref{e10o}) with $k=2$, as well as the fact that $\mathbb{D}_{x} P_n(x;a,b,c$, $d\,|\,q) =\mathcal{D}_q P_n(x;a,b,c,d\,|\,q)$ to obtain
\begin{gather*}P_n(x;a,b,c,d\,|\,q)=\sum_{j=-2}^{2}\gamma_{n+j}\gamma_{n-1+j}b_{n,n+j}P_{n-2+j}(x;aq,bq,cq,dq\,|\,q).\end{gather*}
Substitute the second relation of Lemma \ref{lemma2} into the first to obtain
\begin{gather*}
4P_n(x;a,b,c,d\,|\,q)=C_n (a,b,c,d ) C_{ n-1} (b,aq,c,d ) P_{n-2} (x;bq,aq,c,d\,|\,q )\\
{}+4P_n (x;bq,aq,c,d\,|\,q)- 2(C_n (b,aq,c,d ) +C_n (a,b,c,d )) P_{n-1}(x;bq,aq,c,d\,|\,q ).
\end{gather*}
Substitute $c$ by $cq$, $d$ by $dq$. Permute $a$ and $c$, $b$ and $d$ and use the symmetric property of $P_n$ with respect to its parameters to obtain
\begin{gather*}
4P_n(x;aq,bq,c,d\,|\,q) =4P_n (x;dq,cq,aq,bq\,|\,q ) -2 (C_n (d,cq, aq,bq ) +C_n (c,d,aq,bq ) )\\
{}\times P_{n-1} (x;dq,cq,aq,bq\,|\,q )+C_n (c,d, aq,bq) C_{n-1} (d,cq,aq,bq ) P_{n-2}(x;dq,cq,aq,bq|q ).
\end{gather*}
Take into account this relation in the previous one to obtain the result.
\end{proof}

\begin{Remark} For the continuous $q$-Hermite polynomials $2^nP_n(x)=H_n(x\,|\,q)$ (cf.~\cite[equation~(14.26.1)]{KSL}), the \textit{second structure relation} (\ref{secondSR}) reads as
\begin{gather*}
P_n(x)=\sum_{j=-2}^{2}b_{n,n+j}\gamma_{n+j}\gamma_{n-1+j}P_{n-2+j}(x),
\end{gather*}
for $\mathbb{D}_{x} P_n(x)=\mathcal{D}_qP_n(x)=\gamma_nP_{n-1}(x)$ (cf.\ \cite[equation~(14.26.7)]{KSL}). Therefore $\gamma_{n+2}\gamma_{n+1}b_{n,n+2}=1$ and $b_{n,n+j}=0, j=-2,\dots,1$. Thus, for the monic continuous $q$-Hermite polynomials~$P_n$, the right hand side of~(\ref{secondSR}) is $P_n$. This result is analogous to that of monic Hermite polynomials for the second structure relation~(\ref{e57b}) (cf.~\cite[Table~VI]{Marcellan1994}).
\end{Remark}

\section{Conclusion}
We present another treatment of generalized Bochner theorem and develop two structure relations for classical orthogonal polynomials of the quadratic and $q$-quadratic variable: a~first structure relation that we use to characterize Wilson polynomials, continuous dual Hahn polynomials, Askey--Wilson polynomials and subcases, including limiting cases when one or more parameters tend to~$\infty$, as the family of classical orthogonal polynomials of the quadratic and $q$-quadratic variable; a~second structure relation involving only the divided-difference operator $\mathbb{D}_{x}$, that generalizes the Wilson operator and the Askey--Wilson operator. Our treatment of the gene\-ra\-li\-zed Bochner theorem leads to explicit solutions of the difference equation \cite[equation~(1.3)]{VinetZhedanov}. This work generalizes the result in \cite[Theorem~3.1]{Mourad2003} as well as our previous work (cf.~\cite{MK2017}), where we completed and proved the conjecture by Ismail (cf.\ \cite[equation~(24.7.9)]{Mourad-2005}). Moreover, by completing the work of Koornwinder (cf.~\cite{Koorn}) as shown in~\cite{MK2017} and that of Costas-Santos and Marcell\'{a}n (cf.~\cite{CSM}) in the present paper, we have illustrated that polynomials appearing in the Askey scheme and $q$-Askey scheme \cite{KSL} can be effectively studied by using only the operator~$\mathbb{D}_{x}$.

\appendix\section{Appendix}
In this section, we state and prove some contiguous relations for Wilson polynomials and Askey--Wilson polynomials. To the best of our knowledge, those for Wilson polynomials are new.
\begin{Lemma}\label{lemma1} The monic Wilson polynomials
\begin{gather*}
P_n\big(s^2;a,b,c,d\big)=\frac{W_n\big(s^2;a,b,c,d\big)}{ (-1)^n(a+b+c+d+n-1)_n},\end{gather*} where $W_n\big(s^2;a,b,c,d\big)$ is given by~\eqref{e59}, satisfy the contiguous relations
\begin{gather}
\big(s^2+a^2\big)P_n\big(s^2;a+1,b,c,d\big)=P_{n+1}\big(s^2;a,b,c,d\big)+A_n(a,b,c,d)P_n\big(s^2;a,b,c,d\big),\nonumber\\
\big(s^2+b^2\big)P_n\big(s^2;a,b+1,c,d\big)=P_{n+1}\big(s^2;a,b,c,d\big)+A_n(b,a,c,d)P_n\big(s^2;a,b,c,d\big),\nonumber\\
\big(s^2+c^2\big)P_n\big(s^2;a,b,c+1,d\big)=P_{n+1}\big(s^2;a,b,c,d\big)+A_n(c,b,a,d)P_n\big(s^2;a,b,c,d\big),\nonumber\\
\big(s^2+d^2\big)P_n\big(s^2;a,b,c,d+1\big)=P_{n+1}\big(s^2;a,b,c,d\big)+A_n(d,b,c,a)P_n\big(s^2;a,b,c,d\big),\nonumber\\
P_n\big(s^2;a,b,c,d\big)=P_n\big(s^2;a+1,b,c,d\big)+C_n(a,b,c,d)P_{n-1}\big(s^2;a+1,b,c,d\big)\label{e10k},\\
P_n\big(s^2;a+1,b,c,d\big)=P_n\big(s^2;a+1,b+1,c,d\big)\nonumber\\
\hphantom{P_n\big(s^2;a+1,b,c,d\big)=}{} +C_n(b,a+1,c,d)P_{n-1}\big(s^2;a+1,b+1,c,d\big),\label{e10l}
\end{gather}
where
\begin{gather*}
A_n(a,b,c,d)=\frac{(a+b+c+d+n-1)(a+b+n)(a+c+n)(a+d+n)}{ (a+b+c+d+2n-1)(a+b+c+d+2n)},\\
C_n(a,b,c,d)=\frac{n(b+c+n-1)(b+d+n-1)(c+d+n-1)}{ (a+b+c+d+2n-2)(a+b+c+d+2n-1)}
\end{gather*}
are the coefficients appearing in the three-term recurrence relation {\rm \cite[equation~(9.1.5)]{KSL}} of $P_n\big(s^2;a,b,c,d\big)$.
\end{Lemma}
\begin{proof} For the first relation,
write \begin{gather*}\big(s^2+a^2\big)P_n\big(s^2;a+1,b,c,d\big)=\sum_{j=0}^{n+1}l_jP_j\big(s^2;a,b,c,d\big),\end{gather*}
and use the fact that $\big\{P_n\big(s^2;a,b,c,d\big)\big\}_{n=0}^{\infty}$ is orthogonal with respect to the weight function \begin{gather*}w\big(s^2;a,b,c,d\big)=\left|\frac{\Gamma(a+{\rm i}s)\Gamma(b+{\rm i}s)\Gamma(c+{\rm i}s)\Gamma(d+{\rm i}s)}{ \Gamma(2{\rm i}s)}\right|\end{gather*} on the interval $(0; \infty)$ (cf.~\cite[equation~(9.1.2)]{KSL}), where $\Gamma$ is the gamma function, to obtain
\begin{gather*}
l_j\int_{0}^{\infty}w\big(s^2;a,b,c,d\big)P^2_j\big(s^2;a,b,c,d\big){\rm d}s\\
\qquad{} =\int_{0}^{\infty}w\big(s^2;a,b,c,d\big)\big(s^2+a^2\big)P_n\big(s^2;a+1,b,c,d\big)P_j\big(s^2;a,b,c,d\big){\rm d}s.
\end{gather*}
Use the relation
\begin{gather}\label{wilsonweight}
w\big(s^2;a+1,b,c,d\big)=\big(s^2+a^2\big)w\big(s^2;a,b,c,d\big),
\end{gather}
obtained from the property $\Gamma(z+1)=z\Gamma(z)$, to obtain
\begin{gather*}l_j\int_{0}^{\infty}w\big(s^2;a,b,c,d\big)P^2_j\big(s^2;a,b,c,d\big){\rm d}s\\
\qquad{} =\int_{0}^{\infty}w\big(s^2;a+1,b,c,d\big)P_n\big(s^2;a+1,b,c,d\big)P_j\big(s^2;a,b,c,d\big){\rm d}s=0,\end{gather*}
for $j<n$. That is \begin{gather*}\big(s^2+a^2\big)P_n\big(s^2;a+1,b,c,d\big)=P_{n+1}\big(s^2;a+1,b,c,d\big)+l_{n}P_{n}\big(s^2;a,b,c,d\big).\end{gather*}
Let $s^2=-a^2$ and solve the equation to obtain $l_{n}=A_n(a,b,c,d)$. Since $w\big(s^2;a,b,c,d\big)$ is symmetric with respect to $a$, $b$, $c$ and $d$, and $P_n\big(s^2;a,b,c,d\big)$ is monic, $P_n\big(s^2;a,b,c,d\big)$ is symmetric with respect to $a$, $b$, $c$ and $d$. Using this property, the second, third and fourth relation are deduced from the first. Note that, since the family $\big\{P_n\big(s^2;a+1,b,c,d\big)\big\}_{n=0}^{\infty}$ is orthogonal with respect to $\big(s^2+a^2\big)w\big(s^2;a,b,c,d\big)$ (see~(\ref{wilsonweight})), the polynomial $\big(s^2+a^2\big)$ is nonnegative on $(0, \infty)$ for $\operatorname{Re}(a,b,c,d)>0$ and non-real parameters occur in conjugate pairs (see \cite[p.~186]{KSL}), the first relation of the lemma can be also deduced from \cite[Theorem~2.7.1]{Mourad-2005}.

For (\ref{e10k}), expand $P_n\big(s^2;a+1,b,c,d\big)$ in the basis $P_j\big(s^2;a,b,c,d\big)$; use the fact that $\big\{P_n\big(s^2;a$, $b,c,d\big)\big\}_{n=0}^{\infty}$ is orthogonal with respect to $w\big(s^2;a,b,c,d\big)$ as well as the relation (\ref{wilsonweight}) to obtain
\begin{gather*}
P_n\big(s^2;a,b,c,d\big)=P_{n}\big(s^2;a+1,b,c,d\big)+m_{n-1}P_{n-1}\big(s^2;a+1,b,c,d\big).
\end{gather*}
Apply $\mathbb{D}_{x}$ $n-1$ time to both sides then use the relation, with $k=n-1$,
\begin{gather}\label{nderiv} \mathbb{D}_{x}^kP_n\big(s^2;a,b,c,d\big)=(-n)_kP_{n-k}\left(s^2;a+\frac{k}{ 2},b+\frac{k}{ 2},c+\frac{k}{2},d+\frac{k}{2}\right),\end{gather}
obtained by iterating \cite[equation~(9.1.8)]{KSL}, with \begin{gather*}W_n(s^2;a,b,c,d)=(-1)^n(a+b+c+d+n-1)_nP_n(s^2;a,b,c,d),\end{gather*} and solve the equation with unknown $m_{n-1}$ to obtain $m_{n-1}=C_n(a,b,c,d)$. For the last relation, permute $a$ and $b$ in~(\ref{e10k}) then substitute $a$ for $a+1$ and use the fact that $P_n(s^2;b,a+1,c,d)=P_n(s^2;a+1,b,c,d)$, for $P_n$ is symmetric with respect to its parameters, to obtain the result.
\end{proof}
\begin{Lemma}[\cite{TAK2018}] \label{lemma2}The Askey--Wilson polynomials satisfy the contiguous relations
\begin{gather*}
P_n (x;a,b,c,d\,|\,q )=P_n (x;aq,b,c,d\,|\,q) -\frac{C_n (a,b,c,d;q )}{ 2} P_{n-1} (x;aq,b,c,d\,|\,q ),\\
P_n (x;aq,b,c,d\,|\,q )=P_n (x;aq,bq,c,d\,|\,q) -\frac{C_n (b,aq,c,d )}{ 2} P_{n-1} (x;aq,bq,c,d\,|\,q ),
\end{gather*}
where
\begin{gather*}
C_n(a,b,c,d)={\frac {a \big( 1-{q}^{n} \big) \big( 1-bc{q}^{n-1} \big)\big( 1-bd{q}^{n-1} \big) \big( 1-dc{q}^{n-1} \big) }{ \big(1-abcd{q}^{2 n-2} \big) \big( 1-abcd{q}^{2 n-1} \big) }}
\end{gather*} is the coefficient $C_n$ appearing in the three-term recurrence relation {\rm \cite[equation~(14.1.5)]{KSL}}.
\end{Lemma}
\begin{proof}For the first relation, expand $P_n(x;a,b,c,d\,|\,q)$ in the basis $\{P_j(x;aq,b,c,d\,|\,q)\}$. Use the orthogonality relation \cite[equation~(2.3)]{Askey-1985} as well as the relation $w(x;aq,b,c,d\,|\,q)=\big(1-2ax+a^2\big)w(x;a,b,c,d\,|\,q)$ (cf.\ \cite[p.~16]{Askey-1985}) to obtain \begin{gather*}P_n(x;a,b,c,d\,|\,q)=P_n(x;aq,b,c,d\,|\,q)+t_{n-1}P_{n-1}(x;aq,b,c,d\,|\,q).\end{gather*} Apply $\mathcal{D}_q^{n-1}$ to both sides and take into account the relation \begin{gather}\label{e10o}\mathcal{D}_q^{k}P_{n-1}(x;a,b,c,d\,|\,q)=\gamma_{n}\gamma_{n-1}\cdots \gamma_{n-k-1}P_{n-k}\left(x;a+\frac{k}{ 2},b+\frac{k}{ 2},c+\frac{k}{ 2},d+\frac{k}{ 2}\,|\,q\right),\end{gather} deduced from \cite[equation~(14.1.9)]{KSL}. Solve the equation obtained for the unknown $t_{n-1}$ to get the result. For the second relation, permute $a$ and $b$ in the first one, substitute $a$ by $aq$ and use the fact that~$P_n$ is symmetric with respect to its parameters, that is $P_n(x; a,b,c,d\,|\,q)=P_n(x; b,a,c,d\,|\,q)$ (cf.\ \cite[p.~15]{Askey-1985}), to obtain the result.
\end{proof}

\subsection*{Acknowledgements}The research of MKN was supported by a Vice-Chancellor's Postdoctoral Fellowship from the University of Pretoria. The research by KJ was partially supported by the National Research Foundation of South Africa under grant number 108763. MKN thanks the African Institute for Mathematical Sciences, Muizenberg, South Africa, for their hospitality during his research visit in January 2018 where this paper was completed. We thank the referees for their careful consideration of the manuscript and helpful comments.

\pdfbookmark[1]{References}{ref}
\LastPageEnding


\begin{thebibliography}{99}
\footnotesize\itemsep=0pt

\bibitem{Al-Salam}
Al-Salam W.A., Characterization theorems for orthogonal polynomials, in
 Orthogonal Polynomials ({C}olumbus, {OH}, 1989), \href{https://doi.org/10.1007/978-94-009-0501-6_1}{\textit{NATO Adv. Sci. Inst.
 Ser.~C Math. Phys. Sci.}}, Vol.~294, Kluwer Acad. Publ., Dordrecht, 1990,
 1--24.

\bibitem{Al-Chihara1972}
Al-Salam W.A., Chihara T.S., Another characterization of the classical
 orthogonal polynomials, \href{https://doi.org/10.1137/0503007}{\textit{SIAM~J. Math. Anal.}} \textbf{3} (1972),
 65--70.

\bibitem{AN2006}
\'{A}lvarez Nodarse R., On characterizations of classical polynomials,
 \href{https://doi.org/10.1016/j.cam.2005.06.046}{\textit{J.~Comput. Appl. Math.}} \textbf{196} (2006), 320--337.

\bibitem{Andrew-Askey}
Andrews G.E., Askey R., Classical orthogonal polynomials, in Orthogonal
 Polynomials and Applications ({B}ar-le-{D}uc, 1984), \href{https://doi.org/10.1007/BFb0076530}{\textit{Lecture Notes in
 Math.}}, Vol.~1171, Springer, Berlin, 1985, 36--62.

\bibitem{Askey-1985}
Askey R., Wilson J., Some basic hypergeometric orthogonal polynomials that
 generalize {J}acobi polynomials, \href{https://doi.org/10.1090/memo/0319}{\textit{Mem. Amer. Math. Soc.}} \textbf{54}
 (1985), iv+55~pages.

\bibitem{ARS1995}
Atakishiyev N.M., Rahman M., Suslov S.K., On classical orthogonal polynomials,
 \href{https://doi.org/10.1007/BF01203415}{\textit{Constr. Approx.}} \textbf{11} (1995), 181--226.

\bibitem{bochner1929}
Bochner S., \"{U}ber {S}turm--{L}iouvillesche {P}olynomsysteme,
 \href{https://doi.org/10.1007/BF01180560}{\textit{Math.~Z.}} \textbf{29} (1929), 730--736.

\bibitem{CSM}
Costas-Santos R.S., Marcell\'{a}n F., {$q$}-classical orthogonal polynomials: a
 general difference calculus approach, \href{https://doi.org/10.1007/s10440-009-9536-z}{\textit{Acta Appl. Math.}} \textbf{111}
 (2010), 107--128, \href{https://arxiv.org/abs/math.CA/0612097}{math.CA/0612097}.

\bibitem{datta}
Datta S., Griffin J., A characterization of some {$q$}-orthogonal polynomials,
 \href{https://doi.org/10.1007/s11139-006-0152-5}{\textit{Ramanujan~J.}} \textbf{12} (2006), 425--437.

\bibitem{foupouagnigni2008}
Foupouagnigni M., On difference equations for orthogonal polynomials on
 nonuniform lattices, \href{https://doi.org/10.1080/10236190701536199}{\textit{J.~Difference Equ. Appl.}} \textbf{14} (2008),
 127--174.

\bibitem{GaMar1995}
Garc\'{i}a A.G., Marcell\'{a}n F., Salto L., A distributional study of discrete
 classical orthogonal polynomials, \href{https://doi.org/10.1016/0377-0427(93)E0241-D}{\textit{J.~Comput. Appl. Math.}} \textbf{57}
 (1995), 147--162.

\bibitem{Grunbaum}
Gr\"{u}nbaum F.A., Haine L., The {$q$}-version of a theorem of {B}ochner,
 \href{https://doi.org/10.1016/0377-0427(95)00262-6}{\textit{J.~Comput. Appl. Math.}} \textbf{68} (1996), 103--114.

\bibitem{Mourad2003}
Ismail M.E.H., A generalization of a theorem of {B}ochner, \href{https://doi.org/10.1016/S0377-0427(03)00536-3}{\textit{J.~Comput.
 Appl. Math.}} \textbf{159} (2003), 319--324.

\bibitem{Mourad-2005}
Ismail M.E.H., Classical and quantum orthogonal polynomials in one variable,
 \href{https://doi.org/10.1017/CBO9781107325982}{\textit{Encyclopedia of Mathematics and its Applications}}, Vol.~98, Cambridge
 University Press, Cambridge, 2005.

\bibitem{Alta}
Jooste A., Zeros of {J}acobi, {M}eixner and {K}rawtchouk polynomials, Ph.D.~Thesis, University of Pretoria, 2013, available at
 \url{https://repository.up.ac.za/handle/2263/30787}.

\bibitem{kfk2017}
Kenfack~Nangho M., Foupouagnigni M., Koepf W., On exponential and trigonometric
 functions on nonuniform lattices, \textit{Ramanujan~J.}, {t}o appear.

\bibitem{MK2017}
Kenfack~Nangho M., Jordaan K., A characterization of {A}skey--{W}ilson
 polynomials, \href{https://doi.org/10.1090/proc/14317}{\textit{Proc. Amer. Math. Soc.}}, {t}o appear,
 \href{https://arxiv.org/abs/1711.03349}{arXiv:1711.03349}.

\bibitem{KSL}
Koekoek R., Lesky P.A., Swarttouw R.F., Hypergeometric orthogonal polynomials
 and their {$q$}-analogues, \href{https://doi.org/10.1007/978-3-642-05014-5}{\textit{Springer Monographs in Mathematics}},
 Springer-Verlag, Berlin, 2010.

\bibitem{Koepf1998}
Koepf W., Schmersau D., Representations of orthogonal polynomials,
 \href{https://doi.org/10.1016/S0377-0427(98)00023-5}{\textit{J.~Comput. Appl. Math.}} \textbf{90} (1998), 57--94,
 \href{https://arxiv.org/abs/math.CA/9703217}{math.CA/9703217}.

\bibitem{Koepf2001}
Koepf W., Schmersau D., On a structure formula for classical {$q$}-orthogonal
 polynomials, \href{https://doi.org/10.1016/S0377-0427(00)00577-X}{\textit{J.~Comput. Appl. Math.}} \textbf{136} (2001), 99--107.

\bibitem{Koorn}
Koornwinder T.H., The structure relation for {A}skey--{W}ilson polynomials,
 \href{https://doi.org/10.1016/j.cam.2006.10.015}{\textit{J.~Comput. Appl. Math.}} \textbf{207} (2007), 214--226,
 \href{https://arxiv.org/abs/math.CA/0601303}{math.CA/0601303}.

\bibitem{Magnus1988}
Magnus A.P., Associated {A}skey--{W}ilson polynomials as {L}aguerre--{H}ahn
 orthogonal polynomials, in Orthogonal Polynomials and their Applications
 ({S}egovia, 1986), \href{https://doi.org/10.1007/BFb0083366}{\textit{Lecture Notes in Math.}}, Vol.~1329, Springer,
 Berlin, 1988, 261--278.

\bibitem{Marcellan1994}
Marcell\'{a}n F., Branquinho A., Petronilho J., Classical orthogonal
 polynomials: a functional approach, \href{https://doi.org/10.1007/BF00998681}{\textit{Acta Appl. Math.}} \textbf{34}
 (1994), 283--303.

\bibitem{Maroni93}
Maroni P., Variations around classical orthogonal polynomials. {C}onnected
 problems, \href{https://doi.org/10.1016/0377-0427(93)90319-7}{\textit{J.~Comput. Appl. Math.}} \textbf{48} (1993), 133--155.

\bibitem{Maroni99}
Maroni P., Semi-classical character and finite-type relations between
 polynomial sequences, \href{https://doi.org/10.1016/S0168-9274(98)00137-8}{\textit{Appl. Numer. Math.}} \textbf{31} (1999),
 295--330.

\bibitem{niki1991}
Nikiforov A.F., Suslov S.K., Uvarov V.B., Classical orthogonal polynomials of a
 discrete variable, \href{https://doi.org/10.1007/978-3-642-74748-9}{\textit{Springer Series in Computational Physics}}, Springer-Verlag,
 Berlin, 1991.

\bibitem{TAK2018}
Tcheutia D.D., Jooste A.S., Koepf W., Mixed recurrence equations and
 interlacing properties for zeros of sequences of classical {$q$}-orthogonal
 polynomials, \href{https://doi.org/10.1016/j.apnum.2017.11.003}{\textit{Appl. Numer. Math.}} \textbf{125} (2018), 86--102.

\bibitem{VinetZhedanov}
Vinet L., Zhedanov A., Generalized {B}ochner theorem: characterization of the
 {A}skey--{W}ilson polynomials, \href{https://doi.org/10.1016/j.cam.2006.11.004}{\textit{J.~Comput. Appl. Math.}} \textbf{211}
 (2008), 45--56, \href{https://arxiv.org/abs/0712.0069}{arXiv:0712.0069}.

\end{thebibliography}
\end{document}